\documentclass{article}
\usepackage{stmaryrd,graphicx}
\usepackage{amssymb,mathrsfs,amsmath,amscd,amsthm}
\usepackage{float}
\usepackage{graphicx,float}
\usepackage[all,cmtip]{xy}
\usepackage{amsfonts,txfonts,pxfonts,latexsym,wasysym}

\DeclareMathAlphabet{\mathpzc}{OT1}{pzc}{m}{it}
\newcommand{\tX}{\widetilde{X}}

\newcommand{\simu}{\sim_{\scru}}
\newcommand{\simequ}{\simeq_{\scru}}
\newcommand{\dt}{\Delta_{2}}
\newcommand{\dtr}{|\Delta_{2}|}

\newcommand{\pdtr}{|\partial\Delta_{2}|}
\newcommand{\spu}{\spanier(\mathscr{U},x_0)}
\newcommand{\spv}{\spanier(\mathscr{V},x_0)}

\newcommand{\tspu}{\Pi^{Sp}(\mathscr{U},x_0)}
\newcommand{\tspv}{\Pi^{Sp}(\mathscr{V},x_0)}
\newcommand{\tspw}{\Pi^{Sp}(\mathscr{W},x_0)}
\newcommand{\nulhu}{\nu(\scru,x_0)}
\newcommand{\nulhv}{\nu(\scrv,x_0)}
\newcommand{\tsp}{\Pi^{Sp}}
\newcommand{\spanier}{\pi^{Sp}}
\newcommand{\tspx}{\tsp(X,x_0)}
\newcommand{\spx}{\spanier(X,x_0)}

\newcommand{\pionex}{\pi_{1}(X,x_0)}
\newcommand{\shape}{\check{\pi}_{1}(X,x_0)}

\newcommand{\ox}{\mathcal{O}(X)}

\newcommand{\scru}{\mathscr{U}}
\newcommand{\scrv}{\mathscr{V}}
\newcommand{\scrw}{\mathscr{W}}
\newcommand{\nerveu}{N(\mathscr{U})}
\newcommand{\nervev}{N(\mathscr{V})}

\newcommand{\reva}{\overline{\alpha}}
\newcommand{\revb}{\overline{\beta}}

\newtheorem{theorem}{Theorem}
\newtheorem{lemma}[theorem]{Lemma}
\newtheorem{proposition}[theorem]{Proposition}
\newtheorem{corollary}[theorem]{Corollary}
\newtheorem{definition}[theorem]{Definition}

\newtheorem{example}[theorem]{Example}

\newtheorem{remark}[theorem]{Remark}

\begin{document}

\title{Thick Spanier groups and the first shape group}
\author{Jeremy Brazas and Paul Fabel}
\maketitle

\begin{abstract}
We develop a new route through which to explore $\ker \Psi_{X},$ the kernel
of the $\pi_{1}-$shape group homomorphism determined by a general space $X$,
and establish, for each locally path connected, paracompact Hausdorff space $%
X$, $\ker \Psi_{X}$ is precisely the Spanier group of $X$.
\end{abstract}

\section{Introduction}

It is generally challenging to understand the fundamental group of a locally
complicated space $X.$ A common tactic is to consider the image $\Psi
_{X}(\pi _{1}(X,x_{0}))$ as a subgroup of the first shape homotopy group $%
\check{\pi}_{1}(X,x_{0})$ via the natural homomorphism $\Psi _{X}:\pi
_{1}(X,x_{0})\rightarrow \check{\pi}_{1}(X,x_{0})$ arising from the \v{C}ech
expansion. In particular, if $\Psi _{X}$ is injective, $X$ is said to be $%
\pi _{1}$-shape injective and one gains a characterization of the elements
of $\pi _{1}(X,x_{0})$ as sequences in an inverse limit of fundamental
groups of polyhedra. Thus $\ker \Psi _{X}$ can be thought of as the data of
the fundamental group forgotten when passing to the first shape group. In
this paper, we develop a new lens in which to study $\Psi _{X}$ and identify
new characterizations of $\ker \Psi _{X}$ in terms of familiar subgroups of
the fundamental group.

The Spanier group $\pi ^{Sp}(X,x_{0})$ of a space $X$, introduced in \cite%
{FRVZ11}, is a subgroup of $\pi _{1}(X,x_{0})$ useful for identifying loops
deformable into arbitrarily small neighborhoods \cite{MPT,Virk1,Virk2}.
Moreover the existence of generalized covering maps \cite{BDLM,FZ07,Dydak11} is
intimately related to both $\pi^{Sp}(X,x_0)$ and $\ker \Psi _{X}$.

It is apparently an open question to understand exactly when $\pi
^{Sp}(X,x_{0})=\ker \Psi _{X}$. On the one hand, for a general space $X,$
the inclusion $\pi ^{Sp}(X,x_{0})\subseteq \ker \Psi _{X}$ holds \cite{FZ07}%
. On the other hand, examples show inclusion can be strict if $X$ fails to
be non-locally path connected \cite{FZ07,FRVZ11}. Our main result, Theorem \ref{leftexactsequence}, captures a decent class of spaces such that $\pi
^{Sp}(X,x_{0})=\ker \Psi _{X}$.

\textit{If} $X$ \textit{is a locally path connected paracompact Hausdorff
space, then} $\ker \Psi _{X}$ \textit{is precisely the Spanier group} $\pi
^{Sp}(X,x_{0})$. \newline

In particular if $X$ is a Peano continuum then $\ker \Psi _{X}=\pi
^{Sp}(X,x_{0}).$ A variety of 2-dimensional examples help motivate this
paper's content as follows.

The occurrence of $\pi _{1}$-shape injectivity (i.e. the vanishing of $\ker
\Psi _{X}$) has been studied in a number of cases: For example if $X$ has
Lebesgue covering dimension $\leq 1$ \cite{Eda98,CC06}, if $X$ is a subspace
of a closed surface (or a planar set) \cite{FZ05}, or if $X$ is a
fractal-like tree of manifolds \cite{FG05}, then $\ker \Psi _{X}=1$. The
equality $\pi ^{Sp}(X,x_{0})=\ker \Psi _{X}$ now indicates $X$ is $\pi _{1}$%
-shape injective iff $\pi ^{Sp}(X,x_{0})=1$.

At the other extreme, the equality $\ker \Psi _{X}=\pi _{1}(X,x_{0})$ can
happen if $X$ is a 2-dimensional Peano continuum (e.g. if $X$ is the join of
two cones over the Hawaiian earring). In particular, this equality will
occur if every loop in $X$ is a "small loop" \cite{Virk1,Virk2} or, more
generally, if $X$ is a Spanier space \cite{MPT}. For such spaces,
understanding $\pi _{1}(X,x_{0})$ is equivalent to understanding $\ker \Psi
_{X}.$

Finally, there is a rich \textquotedblleft middle ground\textquotedblright\
where $\ker \Psi _{X}$ lies strictly between $1$ and $\pi _{1}(X,x_{0}).$
Two important such Peano continua appear in \cite{CMRZZ08,FRVZ11}, one of
which is \textquotedblleft homotopically path Hausdorff\textquotedblright\
and one of which is \textquotedblleft strongly homotopically path
Hausdorff\textquotedblright . In both cases, the verification of the
aforementioned properties and the attempts to manufacture generalized
notions of covering spaces, rely heavily on an understanding of $\pi
^{Sp}(X,x_{0})$.

To investigate when $\ker \Psi _{X}=\pi ^{Sp}(X,x_{0})$, we first modify the
familiar definition. The Spanier group with respect to an open cover $%
\mathscr{U}$ (denoted $\pi ^{Sp}(\mathscr{U},x_{0})$) was introduced in
Spanier's celebrated textbook \cite{Spanier66}. Recall $\pi ^{Sp}(X,x_{0})$
is defined as the intersection of the groups $\pi ^{Sp}(\mathscr{U},x_{0})$
over all open covers $\mathscr{U}$. While the usual groups $\pi ^{Sp}(%
\mathscr{U},x_{0})$ are useful for studying covering space theory and its
generalizations, our modified version, the so-called \textit{thick Spanier
group of }$X$ \textit{\ with respect to} $\mathscr{U}$ (denoted $\Pi ^{Sp}(%
\mathscr{U},x_{0})$) is useful for studying the shape homomorphism $\Psi
_{X} $. In particular, the main technical achievement of this paper, Theorem %
\ref{shortexactsequence}, states:\newline

\textit{If} $\mathscr{U}$ \textit{is an open cover consisting of path
connected sets and} $p_{\mathscr{U}}:X\rightarrow |N(\mathscr{U})|$ \textit{%
is a canonical map to the nerve of} $\mathscr{U}$\textit{, then there is a
short exact sequence} 
\begin{equation*}
\xymatrix{ 1 \ar[r] & \tspu \ar[r] & \pi_1(X,x_0) \ar[r]^-{p_{\scru\ast}} &
\pi_{1}(|\nerveu|,U_0) \ar[r] & 1}.
\end{equation*}

The utility of this short exact sequence is the identification of a
convenient set of generators of $\ker p_{\mathscr{U}\ast}$. The same
sequence with the ordinary Spanier group $\pi^{Sp}(\mathscr{U},x_0)$ in
place of $\Pi^{Sp}(\mathscr{U},x_0)$ fails to be exact even in simple cases:
See Example \ref{circle}.

This paper is structured as follows:

In Section 2, we include necessary preliminaries on simplicial complexes and
review the constructions of the \v{C}ech expansion of a space $X$ in terms
of canonical maps $p_{\mathscr{U}}:X\rightarrow |N(\mathscr{U})|$ (where $|N(%
\mathscr{U})|$ is the nerve of an open cover $\mathscr{U}$ of $X$), the
first shape group, and the homomorphism $\Psi _{X}$.

In Section 3, we define and study thick Spanier groups and their
relationship to ordinary Spanier groups. Thick Spanier groups are
constructed by specifying generators represented by loops lying in pairs of
intersecting elements of $\mathscr{U}$; Section 4 is then devoted to
characterizing generic elements of $\Pi^{Sp}(\mathscr{U},x_{0})$ in terms of
a homotopy-like equivalence relation on loops which depends on the given
open cover $\mathscr{U}$.

Section 5 is devoted to a proof of Theorem \ref{shortexactsequence}
guaranteeing the exactness of the above sequence.

Sections 6 and 7 include applications of the above level short exact
sequences: In Section 6, we obtain an exact sequence: 
\begin{equation*}
\xymatrix{ 1 \ar[r] & \spx \ar[r] & \pi_1(X,x_0) \ar[r]^-{\Psi_{X}} & \shape}
\end{equation*}
and thus identify $\ker \Psi_{X}$ and $\pi^{Sp}(X,x_{0})$ for a large class
of spaces (See Theorem \ref{leftexactsequence}). Finally, in Section 7, we
find application in the theory of a topologically enriched version of the
fundamental group. By construction, the first shape group $\check{\pi}%
_{1}(X,x_0)$ is the inverse limit of discrete groups; the fundamental group $%
\pi_{1}(X,x_0)$ inherits the structure of a topological group when it is
given the pullback topology with respect to $\Psi_X$ (the so-called \textit{shape topology}). We apply the above results to identify a convenient basis for the topology of $\pi_{1}(X,x_0)$ and show this topology consists
precisely of the data of the covering space theory of $X$.

\section{Preliminaries and Definitions}

Throughout this paper, $X$ is assumed to be a path connected topological
space with basepoint $x_0$.

\subsection{Simplicial complexes and paths}

We call upon the theory of simplicial complexes in standard texts such as 
\cite{MS82} and \cite{Mu84}. Much of the notation used is in line with these
sources.

If $K$ is an abstract or geometric simplicial complex and $n\geq 0$ is an
integer, $K_n$ denotes the n-skeleton of $K$ and if $v$ is a vertex of $K$, $%
St(v,K)$ and $\overline{St}(v,K)$ denote the open and closed star of the
vertex $v$ respectively. When $K$ is abstract, $|K|$ denotes the geometric
realization. If vertices $v_1,...,v_n$ span an n-simplex of $|K|$, then $%
[v_1,v_2,...,v_n]$ denotes the n-simplex with the indicated orientation.
Finally, for each integer $n\geq 0$, $sd^{n}|K|$ denotes the n-th
barycentric subdivision of $K$.

We frequently make use of the standard abstract 2-simplex $\Delta_{2}$ which
consists of a single 2-simplex and its faces and whose geometric realization
is $|\Delta_{2}|=\{(t_1,t_2)\in \mathbb{R}^{2}|t_1+t_2\leq 1, t_1,t_2\geq 0\}$. The
boundary $\partial\Delta_{2}\subseteq \Delta_{2}$ is the 1-skeleton $%
(\Delta_{2})_1$ and its realization $|\partial\Delta_{2}|$ is homeomorphic
to the unit circle. If $K$ is a geometric subcomplex of a subdivision of $%
|\Delta_{2}|$ containing the origin, then the origin is assumed to be the
basepoint.

We use the following conventions for paths and loops in simplexes and
general spaces. A \textit{path} in a space $X$ is a map $p:[0,1]\rightarrow
X $ from the unit interval. The \textit{reverse path} of $p$ is the path
given by $\overline{p}(t)=p(1-t)$ and the constant path at a point $x\in X$ will
be denoted $c_{x}$. If $p_{1},p_{2}..., p_{n}:[0,1]\rightarrow X$ are paths
in $X $ such that $p_{j}(1)=p_{j+1}(0)$, the concatenation of this sequence
is the unique path $p_{1}\ast p_{2}\ast \dots\ast p_{n}$, sometimes denoted $
\ast _{j=1}^{n}p_{j}$, whose restriction to $\left[ \frac{j-1}{n},\frac{j}{n}
\right]$ is $p_{j}$.

A path $p:[0,1]\to X$ is a \textit{loop} if $p(0)=p(1)$. Quite often it will
be convenient to view a loop as a map $|\partial\Delta_{2}|\to X$ where $%
|\partial\Delta_{2}|$ is identified with $[0,1]/\{0,1\}\cong S^1\subseteq 
\mathbb{R}^{2}$ by an orientation preserving homeomorphism in $\mathbb{R}%
^{2} $. A loop $p:|\partial\Delta_{2}|\to X$ is \textit{inessential} if it
extends to a map $|\Delta_{2}|\to X$ and is \textit{essential} if no such
extension exists. Two loops $p,p^{\prime}:|\partial\Delta_{2}|\to X$ are 
\textit{freely homotopic} if there is a homotopy $H:|\partial\Delta_{2}|%
\times [0,1]\to X$ such that $H(x,0)=p(x)$ and $H(x,1)=p^{\prime}(x)$. If $%
p,p^{\prime}$ are loops based at $x_0\in X$, then they are \textit{homotopic
rel. basepoint} if there is a homotopy $H$ as above such that $%
H(\{(0,0)\}\times [0,1])=x_0$.

We use special notation for edge paths in a geometric simplicial complex $K$%
: If $v_1,v_2$ are vertices in $K$, we identify the oriented 1-simplex $%
[v_1,v_2]$ with the linear path from vertex $v_1$ to $v_2$ on $[v_1,v_2]$.
Similarly, $[v,v]$ denotes the constant path at a vertex $v$. An \textit{%
edge path} is a concatenation $E=[v_0,v_1]\ast [v_1,v_2]\ast \dots \ast
[v_{n-1},v_n]$ of linear or constant paths in the 1-skeleton of $K$. If $%
E(0)=E(1)$, then $E$ is an \textit{edge loop}. It is a well-known fact that
all homotopy classes (rel. endpoints) of paths from $v_0$ to $v_1$ are
represented by edge paths.

\subsection{The \v{C}ech expansion and the first shape group}

We now recall the construction of the first shape homotopy group $\check{\pi}%
_{1}(X,x_0)$ via the \v{C}ech expansion. For more details, see \cite{MS82}.

Let $\mathcal{O}(X)$ be the set of open covers of $X$ direct by refinement.
Similarly, let $\mathcal{O}(X,x_0)$ be the set of open covers with a
distinguished element containing the basepoint, i.e. the set of pairs $(%
\mathscr{U},U_0)$ where $\mathscr{U}\in \mathcal{O}(X)$, $U_0\in \mathscr{U}$
and $x_0\in U_0$. We say $(\mathscr{V},V_0)$ refines $(\mathscr{U},U_0)$ if $%
\mathscr{V}$ refines $\mathscr{U}$ as a cover and $V_0\subseteq U_0$.

The nerve of a cover $(\mathscr{U},U_0)\in\mathcal{O}(X,x_0)$ is the
abstract simplicial complex $N(\mathscr{U})$ whose vertex set is $N(
\mathscr{U})_0=\mathscr{U}$ and vertices $A_0,...,A_n\in \mathscr{U}$ span
an n-simplex if $\bigcap_{i=0}^{n}A_i\neq \emptyset$. The vertex $U_0$ is
taken to be the basepoint of the geometric realization $|N(\mathscr{U})|$.
Whenever $(\mathscr{V},V_0)$ refines $(\mathscr{U},U_0)$, construct a
simplicial map $p_{\mathscr{U}\mathscr{V}}:N(\mathscr{V})\to N(\mathscr{U})$
, called a \textit{projection}, given by sending a vertex $V\in N(\mathscr{V}
)$ to a vertex $U\in \mathscr{U}$ such that $V\subseteq U$. In particular, $
V_0$ must be sent to $U_0$. Any such assignment of vertices extends linearly to a simplicial map. Moreover, the induced map $|p_{\mathscr{U}\mathscr{V}}|:|N(
\mathscr{V})|\to|N(\mathscr{U})|$ is unique up to based homotopy. Thus the
homomorphism $p_{\mathscr{U}\mathscr{V}\ast}:\pi_{1}(|N(\mathscr{V})|,V_0)\to \pi_{1}(|N(\mathscr{U})|,U_0)$ induced on fundamental groups is
independent of the choice of simplicial map.

Recall that an open cover $\scru$ of $X$ is normal if it admits a partition of unity subordinated to $\scru$. Let $\Lambda$ be the subset of $\mathcal{O}(X,x_0)$ (also directed by refinement) consisting of pairs $(\mathscr{U},U_0)$ where $\scru$ is a normal open cover of $X$ and such that there is a partition of unity $\{\phi_{U}\}_{U\in \mathscr{U}}$ subordinated to $\scru$ with $\phi_{U_0}(x_0)=1$. It is well-known that every open cover of a paracompact Hausdorff space $X$ is normal. Moreover, if $(\scru,U_0)\in \mathcal{O}(X,x_0)$, it is easy to refine $(\scru,U_0)$ to a cover $(\scrv,V_0)$ such that $V_0$ is the only element of $\scrv$ containing $x_0$ and therefore $(\scrv,V_0)\in \Lambda$. Thus, for paracompact Hausdorff $X$, $\Lambda$ is cofinal in $\mathcal{O}(X,x_0)$.

The \textit{first shape homotopy group} is the inverse limit 
\begin{equation*}
\check{\pi}_{1}(X,x_0)=\varprojlim\left(\pi_{1}(|N(\mathscr{U})|,U_0),p_{
\mathscr{U}\mathscr{V}\ast},\Lambda\right).
\end{equation*}

Given an open cover $(\mathscr{U},U_0)\in \mathcal{O}(X,x_0)$, a map $p_{\mathscr{U}}:X\to |N(\mathscr{U})|$ is a \textit{(based) canonical map} if $p_{\mathscr{U}}^{-1}(St(U,N(\mathscr{U})))\subseteq U$ for each $U\in \mathscr{U}$ and $p_{\scru}(x_0)=U_0$. Such a canonical map is guaranteed to exist if $(\mathscr{U},U_0)\in
\Lambda $: find a locally finite partition of unity $\{\phi_{U}\}_{U\in \mathscr{U}}$ subordinated to $\scru$ such that $\phi_{U_0}(x_0)=1$. When $U\in \mathscr{U}$ and $x\in U$, determine $p_{\mathscr{U}}(x)$ by requiring its barycentric coordinate belonging to the vertex $U$ of $|N(\mathscr{U})|$ to be $\phi_{U}(x)$. According to this construction, the requirement $\phi_{U_0}(x_0)=1$ gives $p_{\scru}(x_0)=U_0$.

A canonical map $p_{\mathscr{U}}$ is unique up to based homotopy and whenever $(\mathscr{V},V_0)$ refines $(\mathscr{U},U_0)$; the compositions $p_{\mathscr{U}\mathscr{V}}\circ p_{\mathscr{V}}$ and $p_{\mathscr{U}}$ are homotopic as based maps. Therefore the homomorphisms $p_{\mathscr{U}\ast}:\pi_{1}(X,x_0)\to \pi_{1}(|N(\mathscr{U})|,U_0)$ satisfy $p_{\mathscr{U}\mathscr{V}\ast}\circ p_{\mathscr{V}\ast}=p_{\mathscr{U}\ast}$. These homomorphisms
induce a canonical homomorphism 
\begin{equation*}
\Psi_{X}:\pi_{1}(X,x_0)\to \check{\pi}_{1}(X,x_0)\text{ given by }%
\Psi_{X}([\alpha])=\left([p_{\mathscr{U}}\circ \alpha]\right)
\end{equation*}
to the limit.

It is of interest to characterize $\ker\Psi_{X}$ since, when $\ker\Psi_{X}=1$%
, $\check{\pi}_{1}(X,x_0)$ retains all the data in the fundamental group of $%
X$. A space for which $\ker\Psi_{X}=1$ is said to be $\pi_1$-\textit{shape
injective}.

\section{Thick Spanier groups and their properties}

We begin by recalling the construction of Spanier groups \cite%
{FRVZ11,Spanier66}. Let $\widetilde{X}$ denote the set of homotopy classes
(rel. endpoints) of paths starting at $x_0$, i.e. the star of the
fundamental groupoid of $X$ at $x_0$. As in \cite{Spanier66}, multiplication of homotopy classes of paths is taken in the fundamental groupoid of $X$ so that $[\alpha][\beta]=[\alpha\ast\beta]$ when $\alpha(1)=\beta(0)$.

\begin{definition}
\label{spaniergroupdef}\emph{ Let $\mathscr{U}$ be an open cover of $X$.
The \textit{Spanier group of} $X$ \textit{with respect to} $\mathscr{U}$ is
the subgroup of $\pi_{1}(X,x_0)$, denoted $\pi^{Sp}(\mathscr{U},x_0)$,
generated by elements of the form $[\alpha][\gamma][\overline{\alpha}]$ where $%
[\alpha]\in \widetilde{X}$ and $\gamma:[0,1]\to U$ is a loop based at $%
\alpha(1)$ for some $U\in\mathscr{U}$.}
\end{definition}

Note, in the above definition, $[\alpha]$ is left to vary among all homotopy
(rel. endpoint) classes of paths starting at $x_0$ and ending at $%
\gamma(0)=\gamma(1)$. Thus $\pi^{Sp}(\mathscr{U},x_0)$ is a normal subgroup
of $\pi_{1}(X,x_0)$. Also observe that if $\mathscr{V}$ is an open cover of $%
X$ which refines $\mathscr{U}$, then $\pi^{Sp}(\mathscr{V},x_0)\subseteq
\pi^{Sp}(\mathscr{U},x_0)$.

\begin{definition}
\emph{\ The \textit{Spanier group} of $X$ is the intersection 
\begin{equation*}
\pi^{Sp}(X,x_0)=\bigcap_{\mathscr{U}\in \mathcal{O}(X)}\pi^{Sp}(\mathscr{U}%
,x_0)
\end{equation*}
of all Spanier groups with respect to open covers of $X$. Equivalently, $%
\pi^{Sp}(X,x_0)$ is the inverse limit $\varprojlim\pi^{Sp}(\mathscr{U},x_0)$
of the inverse system of inclusions $\pi^{Sp}(\mathscr{V},x_0)\to \pi^{Sp}(%
\mathscr{U},x_0)$ induced by refinement in $\mathcal{O}(X)$. }
\end{definition}

Note the Spanier group of $X$ is a normal subgroup of $\pi_{1}(X,x_0)$.

\begin{remark}
\emph{\ In the previous two definitions, we are actually using the \textit{%
unbased} Spanier group as defined by the authors of \cite{FRVZ11}. These
authors also define the \textit{based} Spanier group by replacing covers
with covers by pointed sets; the two definitions agree when $X$ is locally
path connected. The unbased version is sufficient for the purposes of this
paper since our main results apply to locally path connected spaces. }
\end{remark}

To study $\Psi_{X}$, we require similar but potentially larger versions of
the Spanier groups constructed above.

\begin{definition}
\emph{\ Let $\mathscr{U}$ be an open cover of $X$. The \textit{thick Spanier
group of} $X$ \textit{with respect to} $\mathscr{U}$ is the subgroup of $%
\pi_{1}(X,x_0)$, denoted $\Pi^{Sp}(\mathscr{U},x_0)$, generated by elements
of the form $[\alpha][\gamma_1][\gamma_2][\overline{\alpha}]$ where $[\alpha]\in 
\widetilde{X}$ and $\gamma_1:[0,1]\to U_1$ and $\gamma_2:[0,1]\to U_2$ are
paths for some $U_1,U_2\in \mathscr{U}$. }
\end{definition}

\begin{figure}[H]
\centering \includegraphics[height=1.5in]{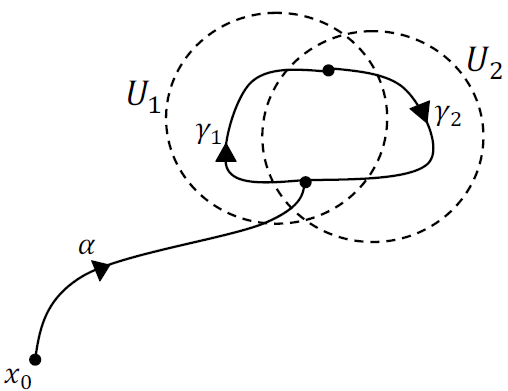}
\caption{A loop representing a generator of the thick Spanier group of $X$
with respect to $\mathscr{U}$.}
\end{figure}

\begin{proposition}
\label{tspanier1} For every open cover $\mathscr{U}$ of $X$, $\Pi^{Sp}(%
\mathscr{U},x_0)$ is a normal subgroup of $\pi_{1}(X,x_0)$ such that $%
\pi^{Sp}(\mathscr{U},x_0)\subseteq \Pi^{Sp}(\mathscr{U},x_0)$. If $%
\mathscr{V}$ is an open cover of $X$ which refines $\mathscr{U}$, then $%
\Pi^{Sp}(\mathscr{V},x_0)\subseteq \Pi^{Sp}(\mathscr{U},x_0)$.
\end{proposition}

\begin{proof}
The subgroup $\tspu$ is normal in $\pionex$ for the same reason $\spu$ is normal. We have $\spu\subseteq \tspu$ since the generators $[\alpha][\gamma_1][\gamma_2][\reva]$ of $\tspu$ where $\gamma_1,\gamma_2$ have image in $U_1=U_2\in\scru$ are precisely the generators of $\spu$. The second statement follows as it does for ordinary Spanier groups.
\end{proof}

\begin{example}
\label{circle} \emph{\ It is not necessarily true that $\pi^{Sp}(\mathscr{U}%
,x_0)=\Pi^{Sp}(\mathscr{U},x_0)$ even in the simplest cases. Suppose $%
\mathscr{U}=\{U_1,U_2\}$ is an open cover of the unit circle $X=S^1$
consisting of two connected intervals $U_1,U_2$ such that $U_1\cap U_2$ is
the disjoint union of two connected intervals. Since both $U_1,U_2$ are
simply connected, we have $\pi^{Sp}(\mathscr{U},x_0)=1$. On the other hand, $%
\Pi^{Sp}(\mathscr{U},x_0)$ contains a generator of $\pi_{1}(S^1,x_0)$ and
thus $\Pi^{Sp}(\mathscr{U},x_0)=\pi_{1}(S^1,x_0)$. }
\end{example}

\begin{remark}
\label{naturalityofspanier}\emph{\ Spanier groups and thick Spanier groups
with respect to covers are natural in the following sense: If $f:X\to Y$ is
a map such that $f(x_0)=y_0$ and $\mathscr{W}$ is open cover of $Y$, then $
f^{-1}\mathscr{W}=\{f^{-1}(W)|W\in\mathscr{W}\}$ is an open cover of $X$.
Observe if $f_{\ast}:\pi_{1}(X,x_0)\to\pi_{1}(Y,y_0)$ is the homomorphism
induced on fundamental groups, then $f_{\ast}(\pi^{Sp}(f^{-1}\mathscr{W}%
,x_0))\subseteq \pi^{Sp}(\mathscr{W},y_0)$ and $f_{\ast}(\Pi^{Sp}(f^{-1}%
\mathscr{W},x_0))\subseteq \Pi^{Sp}(\mathscr{W},y_0)$.}
\end{remark}

\begin{definition}
\emph{\ The \textit{thick Spanier group} of $X$ is the intersection 
\begin{equation*}
\Pi^{Sp}(X,x_0)=\bigcap_{\mathscr{U}\in\mathcal{O}(X)}\Pi^{Sp}(\mathscr{U}%
,x_0)
\end{equation*}
of all thick Spanier groups with respect to open covers of $X$. Equivalently 
$\Pi^{Sp}(X,x_0)$ is the inverse limit $\varprojlim \Pi^{Sp}(\mathscr{U}%
,x_0) $ of the inverse system of inclusions $\Pi^{Sp}(\mathscr{V},x_0)\to
\Pi^{Sp}(\mathscr{U},x_0)$ induced by refinement in $\mathcal{O}(X)$. }
\end{definition}

The following proposition follows directly from previous observations.

\begin{proposition}\label{spanierinclusion}
\label{tspanier2} For every space $X$, $\pi^{Sp}(X,x_0)\subseteq
\Pi^{Sp}(X,x_0)$. If every open cover $\mathscr{U}$ of $X$ admits a
refinement $\mathscr{V}$ such that $\Pi^{Sp}(\mathscr{V},x_0)\subseteq
\pi^{Sp}(\mathscr{U},x_0)$, then $\Pi^{Sp}(X,x_0)=\pi^{Sp}(X,x_0)$.
\end{proposition}

We apply Proposition \ref{tspanier2} to find two conditions (those in
Proposition \ref{pcintersections} and Theorem \ref{paracompact1})
guaranteeing the equality of the Spanier group and thick Spanier group.

\begin{proposition}
\label{pcintersections} If $\mathscr{V}$ is an open cover of $X$ such that $%
V\cap V^{\prime}$ is path connected (or empty) for every pair $V,V^{\prime}\in \mathscr{V}$, then $\Pi^{Sp}(\mathscr{V},x_0)=\pi^{Sp}(\mathscr{V},x_0)$. Consequently, if every open cover $\mathscr{U}$ of $X$ has an open refinement $\mathscr{V}$ with this property, then $\Pi^{Sp}(X,x_0)=\pi^{Sp}(X,x_0)$.
\end{proposition}

\begin{proof}
Suppose $\scrv$ is as in the statement of the proposition. It suffices to show $\tspv\subseteq \spv$. Let $g=[\alpha][\gamma_1][\gamma_2][\reva]$ be a generator of $\tspv$ where $[\alpha]\in \tX$, and $\gamma_i:[0,1]\to V_i$ for $V_i\in \scrv$. Since, by assumption, $V_{1}\cap V_{2}$ is path connected and the points $x_1=\alpha(1)$ and $x_2=\gamma_{1}(1)=\gamma_{2}(0)$ both lie in $V_{1}\cap V_{2}$, there is a path $\beta:[0,1]\to  V_{1}\cap V_{2}$ from $x_1$ to $x_2$. Note $g_1=[\alpha][\gamma_1\ast\revb][\reva]$ and $g_2=[\alpha][\beta\ast\gamma_2][\reva]$ are generators of $\spv$. Since $g=g_1g_2$, it follows that $g\in \spv$. The second statement in the current proposition now follows from Proposition \ref{tspanier2}.
\end{proof}

\begin{example}
\label{starcover} \emph{
Let $K$ be a geometric simplicial complex and $v\in K$ be a vertex. It is well-known that the open star $St(v,K)$ is contractible (and thus path connected). Therefore $\mathscr{S}=\{St(v,K)|v\in K\}$ is an open cover of $K$ such that $\pi^{Sp}(\mathscr{S},v_{0})=1$.\\
\indent Moreover, if $St(v_{1},K)\cap St(v_{2},K)\neq \emptyset $, then the
open 1-simplex $(v_{1},v_{2})\subset St(v_{1},K)\cap St(v_{2},K),$ and hence
there is a canonical strong deformation from $St(v_{1},K)\cap St(v_{2},K)$ onto $(v_{1},v_{2})$. For a given point $p\in St(v_{1},K)\cap St(v_{2},K),$ fix the coefficients of $v_{1}$ and $v_{2}$ while linearly contracting to $0$ the remaining barycentric coordinates of $p$. Hence, since $(v_{1},v_{2})$ is contractible, $St(v_{1},K)\cap St(v_{2},K)$ is contractible and, in particular, path connected. Thus it follows from Proposition \ref{pcintersections} that $\Pi^{Sp}(\mathscr{S},v_{0})=\pi ^{Sp}(\mathscr{S},v_{0})=1$. }
\end{example}

It is known that $\pi^{Sp}(X,x_0)\subseteq \ker\Psi_{X}$ for any space $X$:
see Proposition 4.8 of \cite{FZ07}. Thus the Spanier group of $X$ vanishes
whenever $X$ is $\pi_1$-shape injective. We use Proposition \ref{pcintersections} to prove the analogous inclusion for the thick Spanier group; note the proof is similar to that in \cite{FZ07}.

\begin{proposition}
\label{inclusion1} For any space $X$, $\Pi^{Sp}(X,x_0)\subseteq \ker\Psi_{X}$.
\end{proposition}

\begin{proof}
Let $g\in \tspx$ and suppose $(\scru,U_0)\in\Lambda$. It suffices to show $g\in \ker p_{\scru\ast}$. Recall that for each $U\in \scru$, we have $p_{\scru}^{-1}(St(U,\nerveu))\subseteq U$. Therefore $\scrv=\{p_{\scru}^{-1}(St(U,\nerveu))|U\in \scru\}$ is an open cover of $X$ which refines $\scru$. Since $g\in\tspv$ by assumption, $g$ is a product of generators of the form $[\alpha][\gamma_1][\gamma_2][\reva]$ where $[\alpha]\in \tX$ and $\gamma_i:[0,1]\to p_{\scru}^{-1}(St(U_i,\nerveu))$ for $U_i\in \scru$, $i=1,2$. Note $p_{\scru}\circ\alpha$ is a path in $|\nerveu|$ based at $U_0$ with $p_{\scru}(\alpha(1))\in St(U_1,\nerveu)\cap St(U_2,\nerveu)$ and $p_{\scru}\circ\gamma_i:[0,1]\to St(U_i,\nerveu)$ for $i=1,2$. Thus if $\mathscr{S}=\{St(U,\nerveu)|U\in\scru\}$ is the cover of $|\nerveu|$ by open stars, then \[p_{\scru\ast}([\alpha][\gamma_1][\gamma_2][\reva])=[p_{\scru}\circ\alpha] [p_{\scru}\circ\gamma_1][p_{\scru}\circ\gamma_2][\overline{p_{\scru}\circ\alpha}]\]is a generator of $\tsp(\mathscr{S},U_0)$. But it is observed in Example \ref{starcover} that $\tsp(\mathscr{S},U_0)=1$. Since generators multiplying to $g$ lie in the subgroup $\ker p_{\scru\ast}$ of $\pionex$, so does $g$.
\end{proof}

In the non-locally path connected case, the inclusion $\Pi^{Sp}(X,x_0)%
\subseteq \ker\Psi_{X}$ need not be equality; a counterexample is given later in Example \ref{example1}.

The following theorem calls upon the theory of paracompact spaces. Recall if 
$\mathscr{U}$ is an open cover of $X$ and $x\in X$, the \textit{star of} $x$ 
\textit{with respect to} $\mathscr{U}$ is the union $St(x,\mathscr{U}%
)=\bigcup\{U\in \mathscr{U}|x\in U\}$. A \textit{barycentric refinement} of $%
\mathscr{U}$ is an open refinement $\mathscr{V}$ of $\mathscr{U}$ such that
for each $x\in X$, there is a $U\in \mathscr{U}$ with $St(x,\mathscr{V}%
)\subseteq U$. It is known that a $T_1$ space is paracompact if and only if
every open cover has an open barycentric refinement \cite[5.1.12]{Eng89}.

\begin{theorem}
\label{paracompact1} If $X$ is $T_1$ and paracompact, then $%
\Pi^{Sp}(X,x_0)=\pi^{Sp}(X,x_0)$.
\end{theorem}

\begin{proof}
We apply the second statement of Proposition \ref{tspanier2}. Given an open cover $\scru$ of $X$ take $\scrv$ to be a barycentric refinement of $\scru$. Suppose $g=[\alpha][\gamma_1][\gamma_2][\reva]$ is a generator of $\tspv$ where $[\alpha]\in \tX$ and $\gamma_i:[0,1]\to V_i$ for $V_i\in \scrv$, $i=1,2$. Since $\alpha(1)\in V_1\cap V_2$ and $\scrv$ is chosen to be a barycentric refinement, we have $V_1\cup V_2\subseteq St(\alpha(1),\scrv)\subseteq U$ for some $U\in \scru$. Note $\gamma_{1}\ast\gamma_{2}$ is a loop in $U$ based at $\alpha(1)$ and $g=[\alpha][\gamma_{1}\ast\gamma_{2}][\reva]$. Thus $g\in \spu$ and the inclusion $\tspv\subseteq \spu$ follows.
\end{proof}

\begin{corollary}
If $X$ is metrizable, then $\Pi^{Sp}(X,x_0)=\pi^{Sp}(X,x_0)$.
\end{corollary}

\section{Thick Spanier groups and $\mathscr{U}$-homotopy}

\label{sectionuhomotopy} While the definition of the thick Spanier group $%
\Pi^{Sp}(\mathscr{U},x_0)$ in terms of generators is often convenient, it is desirable to characterize generic elements. We do this by replacing
homotopies as continuous deformations of paths with finite step homotopies
through open covers. This approach is closely related to that of Dugundji in 
\cite{Dug50} and differs from the notions of $\mathscr{U}$-homotopy in \cite%
{Kennison89} and \cite{MS82}.

Fix an open cover $\mathscr{U}$ of $X$ and define a relation $\sim_{%
\mathscr{U}}$ on paths in $X$. If $\alpha,\beta:[0,1]\to X$ are paths in $X$%
, let $\alpha\sim_{\mathscr{U}}\beta$ when $\alpha(0)=\beta(0)$, $%
\alpha(1)=\beta(1)$, and there are partitions $0=a_0<a_1<\dots < a_n=1$ and $%
0=b_0<b_1<\dots < b_n=1$ and a sequence $U_1,...,U_n\in \mathscr{U}$ of open
neighborhoods such that $\alpha([a_{i-1},a_{i}])\subseteq U_i$ and $%
\beta([b_{i-1},b_{i}])\subseteq U_i$ for $i=1,...,n$. Note if $\mathscr{V}$ is an open refinement of $\mathscr{U}$, then $\alpha\sim_{%
\mathscr{V}}\beta$ $\Rightarrow$ $\alpha\sim_{\mathscr{U}}\beta$.

\begin{definition}
\emph{\ Two paths $\alpha$ and $\beta$ in $X$ are said to be $\mathscr{U}$%
\textit{-homotopic} if $\alpha(0)=\beta(0)$, $\alpha(1)=\beta(1)$, and there
is a finite sequence $\alpha=\gamma_0,\gamma_1,...,\gamma_{n}=\beta$ such
that $\gamma_{j-1}\sim_{\mathscr{U}}\gamma_{j}$ for each $j=1,2,...,n$. The
sequence $\gamma_0,\gamma_1,...,\gamma_{n}$ is called a $\mathscr{U}$\textit{%
-homotopy from }$\alpha$ \textit{to} $\beta$. If a loop $\alpha$ is $%
\mathscr{U}$-homotopic to the constant path, we say it is \textit{null-}$%
\mathscr{U}$\textit{-homotopic}. }
\end{definition}

Note $\mathscr{U}$-homotopy defines an equivalence relation $\simeq_{%
\mathscr{U}}$ on the set of paths in $X$.

\begin{lemma}
\label{homotopicimpliesuhomotopic} If two paths are homotopic (rel.
endpoints), then they are $\mathscr{U}$-homotopic.
\end{lemma}

\begin{proof}
Suppose $\alpha,\beta$ are paths in $X$ with $\alpha(0)=\beta(0)$ and $\alpha(1)=\beta(1)$ and $H:[0,1]\times [0,1]\to X$ is a homotopy such that $H(s,0)=\alpha(s)$, $H(s,1)=\beta(s)$, $H(0,t)=\alpha(0)$, and $H(1,t)=\beta(1)$ for all $s,t\in [0,1]$. Find an integer $n\geq 1$ such that for each square $S_{i,j}=\left[\frac{i-1}{n},\frac{i}{n}\right]\times \left[\frac{j-1}{n},\frac{j}{n}\right]$, $i,j\in \{1,2,...,n\}$, $H(S_{i,j})\subseteq U_{i,j}$ for some $U_{i,j}\in \scru$. Let $\gamma_j$ be the path which is the restriction of $H$ to $[0,1]\times \{j/n\}$. We have $\alpha=\gamma_0$, $\beta=\gamma_n$, and $\gamma_{j-1}\simu\gamma_{j}$ for each $j=1,2,...,n$ and thus $\alpha\simequ\beta$.
\end{proof}
Observe that $\mathscr{U}$-homotopy respects inversion and concatenation of paths in the sense that if $\alpha \simeq_{\mathscr{U}}\beta$ and $\alpha ^{\prime
}\simeq_{\mathscr{U}}\beta ^{\prime }$ then $\reva \simeq_{\mathscr{U}%
}\revb$ and $\alpha\ast\alpha ^{\prime}\simeq_{\mathscr{U}}\beta\ast
\beta ^{\prime}$ when the concatenations are defined.

Since inessential loops are null-$\mathscr{U}$-homotopic by Lemma \ref%
{homotopicimpliesuhomotopic}, the set $\nu(\mathscr{U},x_0)$ of null-$%
\mathscr{U}$-homotopic loops is a subgroup of $\pi_{1}(X,x_0)$. In fact, $%
\nu(\mathscr{U},x_0)$ is normal in $\pi_{1}(X,x_0)$ since, if $\alpha$ is
null-$\mathscr{U}$-homotopic and $[\beta]\in \pi_{1}(X,x_0)$, then 
\begin{equation*}
\beta\ast\alpha\ast \revb\simeq_{\mathscr{U}}\beta\ast c_{x_0} \ast
\revb \simeq_{\mathscr{U}} c_{x_0}.
\end{equation*}
We now compare $\Pi^{Sp}(\mathscr{U},x_0)$ and $\nu(\mathscr{U},x_0)$ as
subgroups of $\pi_{1}(X,x_0)$.

\begin{lemma}
\label{uhomotopy} Suppose $\mathscr{U}$ is an open cover of $X$ consisting
of path connected sets. If $\alpha$ and $\beta$ are $\mathscr{U}$-homotopic
paths with $\alpha(0)=x_0=\beta(0)$, then $[\alpha\ast\revb]\in
\Pi^{Sp}(\mathscr{U},x_0)$.
\end{lemma}

\begin{proof}
First, consider the case when $\alpha\simu\beta$. Take open neighborhoods $U_1,...,U_n\in \scru$ and partitions $0=a_0<a_1<\dots < a_n=1$ and $0=b_0<b_1<\dots < b_n=1$ such that $\alpha([a_{i-1},a_{i}])\subseteq U_i$ and $\beta([b_{i-1},b_{i}])\subseteq U_i$. Since $U_i$ is path connected, there are paths $\epsilon_i:[0,1]\to U_i$ from $\alpha(a_i)$ to $\beta(b_i)$ for $i=1,...,n-1$. Additionally, there are paths $\delta_i:[0,1]\to U_{i+1}$ from $\alpha(a_i)$ to $\beta(b_i)$ for $i=1,...,n-1$. Let $\delta_0$ be the constant path at $x_0$ and $\epsilon_n$ be the constant path at $\alpha(1)=\beta(1)$. Observe $\epsilon_i\ast \overline{\delta_{i}}$ is a loop in $U_{i}\cup U_{i+1}$ based at $\alpha(a_i)$ for $i=1,...,n-1$.

Let $\alpha_i$ and $\beta_i$ be the paths given by restricting $\alpha$ and $\beta$ to $[a_{i-1},a_{i}]$ and $[b_{i-1},b_{i}]$ respectively. Note the following paths $g_i$ are generators of $\spu$:
\begin{enumerate}
\item $\displaystyle g_1=[c_{x_0}][\alpha_1\ast\epsilon_1\ast\overline{\beta_{1}}][\overline{c_{x_0}}]$
\item $\displaystyle g_i=[\beta_1\ast \beta_2\ast\dots\ast\beta_{i-1}] [\overline{\delta_{i-1}}\ast\alpha_{i}\ast \epsilon_{i}\ast \overline{\beta_{i}}] [\overline{\beta_1\ast \beta_2\ast\dots\ast\beta_{i-1}}]$ for $i=2,...,n-1$
\item $\displaystyle g_n=[\beta_1\ast \dots\ast\beta_{n-1}][\overline{\delta_{n-1}}\ast \alpha_n \ast \overline{\beta_{n}}][\overline{\beta_1\ast \dots\ast\beta_{n-1}}]$
\end{enumerate}
Since $\spu\subseteq \tspu$, we have $g_i\in \tspu$ for $i=1,...,n$. Additionally, for $i=1,...,n-1$, \[h_i=[\beta_1\ast \dots \ast \beta_i][\overline{\epsilon_{i}}][\delta_i][\overline{\beta_1\ast \dots \ast \beta_i}]\]is a generator of $\tspu$. Note \[[\alpha\ast\revb]=[\alpha_1][\alpha_2]\dots [\alpha_n][\overline{\beta_{n}}]\dots [\overline{\beta_{2}}][\overline{\beta_{1}}]= g_1h_1g_2h_2\dots g_{n-1}h_{n-1}g_n\]and thus $[\alpha\ast\revb]\in \tspu$.

For the general case, suppose there is a $\scru$-homotopy $\alpha=\gamma_0,\gamma_1,...,\gamma_{n}=\beta$ where $\gamma_{j-1}\simu\gamma_{j}$ for each $j=1,2,...,n$. The first case gives $[\gamma_{i-1}\ast\overline{\gamma_{i}}]\in \tspu$ for each $i=1,...,n$ and thus \[[\alpha\ast\revb]=[\gamma_{0}\ast\overline{\gamma_{1}}] [\gamma_{1}\ast\overline{\gamma_{2}}]\dots [\gamma_{n-1}\ast\overline{\gamma_{n}}]\in \tspu.\]
\end{proof}

\begin{theorem}
\label{characterizethickspangrp} For every open cover $\mathscr{U}$, $%
\Pi^{Sp}(\mathscr{U},x_0)\subseteq \nu(\mathscr{U},x_0)$. Moreover, if $%
\mathscr{U}$ consists of path connected sets, then $\Pi^{Sp}(\mathscr{U}%
,x_0)=\nu(\mathscr{U},x_0)$.
\end{theorem}

\begin{proof}
Since $\nulhu$ is a subgroup of $\pionex$, the inclusion  $\tspu\subseteq \nulhu$ follows from showing $\nulhu$ contains the generators of $\tspu$. Suppose $[\beta]=[\alpha][\gamma_1][\gamma_2][\reva]$ is such a generator where $\gamma_i$ has image in $U_i\in \scru$. We claim if $\beta=\alpha\ast\gamma_1\ast \gamma_2\ast \reva$, then $\beta\simu \alpha\ast\reva$.

Find a sequence of neighborhoods $V_1,...,V_n\in \scru$ and a partition $0=s_0<s_1<\dots <s_n=1$ such that $\alpha([s_{j-1},s_{j}])\subseteq V_j$. Since $\alpha(1)\in U_1\cap U_2$ we may assume $\alpha([s_{n-1},1])\subseteq U_1\cap U_2$.

For $c>0$, let $L_c:[0,1]\to [0,c]$ be the order-preserving, linear homeomorphism. Define a partition $0=t_0<t_1<\dots <t_{2n}=1$ by $t_k=L_{1/2}(s_k)$ if $0\leq k\leq n-1$, $t_n=1/2$, and $t_k=1-L_{1/2}(s_{2n-k})$ for $n+1\leq k\leq 2n$. Rewrite the sequence $V_1,...,V_{n-1},U_1,U_2,V_{n-1},...,V_1$ in $\scru$ as $W_1,...,W_{2n}$. Note $\alpha\ast\reva([t_{k-1},t_k])\subseteq W_k$ for $k=1,...,2n$.

Define a partition $0=r_0<r_1<\dots <r_{2n}=1$ by $r_k=L_{1/4}(s_k)$ if $0\leq k\leq n-1$, $r_n=1/2$, and $r_k=1-L_{1/4}(s_{2n-k})$ for $n+1\leq k\leq 2n$. Note $\beta([r_{k-1},r_k])\subseteq W_k$ for $k=1,...,2n$ and thus $\beta\simu\alpha\ast\reva$. Since $\alpha\ast\reva$ is an inessential loop, $\beta$ is null-$\scru$-homotopic by Lemma \ref{homotopicimpliesuhomotopic}.

In the case that $\scru$ consists of path connected sets, note that if $\alpha$ is a null-$\scru$-homotopic loop based at $x_0$, then $[\alpha]=[\alpha\ast \overline{c_{x_0}}]\in \tspu$ by Lemma \ref{uhomotopy}. This gives the inclusion $\nu(\mathscr{U},x_0)\subseteq \Pi^{Sp}(\mathscr{U},x_0)$.
\end{proof}

\section{A \v{C}ech-Spanier short exact sequence}

\label{mainsection} The following theorem, which identifies the kernel of
the homomorphism $p_{\mathscr{U}\ast}:\pi_{1}(X,x_0)\to \pi_{1}(|N(%
\mathscr{U})|,U_0)$ for certain covers $\mathscr{U}$, is the main technical
achievement of this paper.

\begin{theorem}
\label{shortexactsequence} Suppose $(\mathscr{U},U_0)\in \Lambda$ is such that $\mathscr{U}$ consists of path connected sets and $p_{\mathscr{U}}:X\to |N(\mathscr{U})|$ is a canonical map. The sequence 
\begin{equation*}
\xymatrix{ 1 \ar[r] & \tspu \ar[r]^-{i_{\scru}} & \pionex
\ar[r]^-{p_{\scru\ast}} & \pi_{1}(|\nerveu|,U_0) \ar[r] & 1}
\end{equation*}
where $i_{\mathscr{U}}$ is inclusion is exact.
\end{theorem}

The above theorem is proved in the following three subsections in which the
hypotheses of the statement are assumed; exactness follows from Theorems \ref{exactness1}, \ref{exactness2} and \ref{exactness3}. Note if $(\mathscr{V},V_0)\in\Lambda$ is a normal refinement of $(\mathscr{U},U_0)$, where $\mathscr{V}$ also consists of path connected sets and $p_{\mathscr{V}}:X\to |N(\mathscr{V})|$ is a canonical map, then there is a morphism 
\begin{equation*}
\xymatrix{ 1 \ar[r] & \tspv \ar[d]^{i_{\scru\scrv}} \ar[r]^{i_{\scrv}}
\ar[r] & \pionex \ar[d]^{id} \ar[r]^-{p_{\scrv\ast}} &
\pi_{1}(|\nervev|,V_0) \ar[d]^{p_{\scru\scrv\ast}} \ar[r] & 1\\ 1 \ar[r] &
\tspu \ar[r]_{i_{\scru}} & \pionex \ar[r]_-{p_{\scru\ast}} &
\pi_{1}(|\nerveu|,U_0) \ar[r] & 1}
\end{equation*}
of short exact sequences of groups where the left square consists of three
inclusion maps and (at least) one identity.

\begin{corollary}
\label{sesofinversesystem} Suppose $X$ is paracompact Hausdorff and $\Lambda
^{\prime}$ is a directed subset of $\Lambda$ such that if $(\scru,U_0)\in \Lambda '$, then $\mathscr{U}$ contains only path connected sets. There are level morphisms of inverse systems $i:(\Pi^{Sp}(\scru,x_0),i_{\scru\scrv},\Lambda ^{\prime})\to
(\pi_{1}(X,x_0),id,\Lambda ^{\prime})$ and $p:(\pi_{1}(X,x_0),id,\Lambda
^{\prime})\to (\pi_{1}(|N(\mathscr{U})|,U_0),p_{\scru\scrv\ast},\Lambda ^{\prime})$ such that 
\begin{equation*}
\xymatrix{ 1 \ar[r] & \tspu \ar[r]^-{i_{\scru}} & \pionex
\ar[r]^-{p_{\scru\ast}} & \pi_{1}(|\nerveu|,U_0) \ar[r] & 1}
\end{equation*}
is exact for each $(\scru,U_0)\in \Lambda ^{\prime}$. Moreover, 
\begin{equation*}
\xymatrix{ 1 \ar[r] & (\tspu,i_{\scru\scrv},\Lambda ') \ar[r]^-{i} &
(\pionex,id,\Lambda ') \\& \ar[r]^-{p} &
(\pi_{1}(|\nerveu|,U_0),p_{\scru\scrv\ast},\Lambda ') \ar[r] & 1}
\end{equation*}
is an exact sequence in the category of pro-groups.
\end{corollary}

\begin{proof}
The level sequences, i.e. the sequences for each element of $\Lambda '$, are exact by Theorem \ref{shortexactsequence}. This gives rise to the exact sequence in the category of pro-groups \cite[Ch. II, \textsection 2.3, Theorem 10]{MS82}.
\end{proof}
Theorems \ref{characterizethickspangrp} and \ref{shortexactsequence} combine
to give a characterization of $\ker p_{\mathscr{U}\ast}$ in terms of the $%
\mathscr{U}$-homotopy of the previous Section.

\begin{corollary}
If $\mathscr{U}$ consists of path connected sets and $\alpha:[0,1]\to X$ is
a loop based at $x_0$, then $[\alpha]\in \ker p_{\mathscr{U}\ast}$ if and
only if $\alpha$ is null-$\mathscr{U}$-homotopic.
\end{corollary}

\subsection{Surjectivity of $p_{\mathscr{U}\ast}$}

\begin{lemma}
\label{homotopicpaths} Let $K$ be a simplicial complex and $%
L,L^{\prime}:[0,1]\to |K|$ be loops where $L(0)=z_1=L(1)$ and $%
L^{\prime}(0)=z_2=L^{\prime}(1)$. If there are vertices $v_1,v_2,...,v_m\in
K $ and partitions 
\begin{equation*}
0=s_0<s_1<s_2<\dots <s_m=1 \text{ and }0=t_0<t_1<t_2<\dots <t_m=1
\end{equation*}
such that $L([s_{k-1},s_k])\cup L^{\prime}([t_{k-1},t_k])\subset St(v_k,K)$
for each $k=1,2,...,m$, then $L$ and $L^{\prime}$ are freely homotopic in $%
|K|$. If $z_1=z_2$, then $L$ and $L^{\prime}$ are homotopic rel. basepoint.
\end{lemma}

\begin{proof}
Since a non-empty intersection of two stars in a simplicial complex is path connected, there is path $\gamma_k:[0,1]\to St(v_{k},K)\cap St(v_{k+1},K)$ from $L(s_{k})$ to $L'(t_k)$ for $k=1,...,m-1$. Let $\gamma_0=\gamma_{m}$ be a path in $St(v_m,K)\cap St(v_1,K)$ from $z_1$ to $z_2$. Each loop $L|_{[s_{k-1},s_k]}\ast \gamma_{k}\ast \overline{L'|_{[t_{k-1},t_k]}}\ast \overline{\gamma_{k-1}}$, $k=1,...,m$ has image in the star of a vertex and is therefore inessential. Thus $L$ and $L'$ are freely homotopic. If $z_1=z_2$, the same argument produces a basepoint-preserving homotopy when we take $\gamma_0=\gamma_m$ to be the constant path at $z_1=z_2$.
\end{proof}

\begin{lemma}
\label{surjectivity} If $\mathscr{U}$ consists of path connected sets and $%
E=[U_1,U_2]\ast [U_2,U_3]\ast \dots \ast[U_{n-1},U_n]$ is an edge loop in $%
|N(\mathscr{U})|$ such that $U_1=U_0=U_n$, then there is a loop $%
\alpha:[0,1]\to X$ based at $x_0$ such that $L=p_{\mathscr{U}}\circ \alpha$
and $E$ are homotopic rel. basepoint.
\end{lemma}

\begin{proof}
Since each $U_j$ is path connected and $U_j\cap U_{j+1}\neq \emptyset$, there is a loop $\alpha:[0,1]\to X$ based at $x_0$ such that $\alpha\left(\left[\frac{j-1}{n},\frac{j}{n}\right]\right)\subseteq U_j$ for each $j=1,...,n$. We claim $E$ is homotopic to $L=p_{\scru}\circ \alpha$.

Recall $V_U=p_{\scru}^{-1}(St(U,\nerveu))\subseteq U$ for each $U\in\scru$ so that $\scrv=\{V_U|U\in\scru\}$ is an open refinement of $\scru$. Consequently, for each $j=1,...,n$, there is a subdivision $\frac{j-1}{n}=s_{j}^{0}<s_{j}^{1}<\dots <s_{j}^{m_j}=\frac{j}{n}$ and a sequence $W^{1}_{j},...,W^{m_j}_{j}\in \scru$ such that $\alpha([s^{k-1}_{j},s^{k}_{j}])\subseteq V_{W_{j}^{k}}\subseteq W_{j}^{k}$ for $k=1,...,m_j$. We may assume $W_{j}^{m_j}=W_{j+1}^{1}$ for $j=1,...,n-1$ and, since $x_0\in V_{U_0}$, we may assume $W_{1}^{1}=U_0=W_{n}^{m_n}$. Note that $L([s^{k-1}_{j},s^{k}_{j}])\subseteq St(W_{j}^{k},\nerveu)$.\\

Define an edge loop $L'$ based at the vertex $U_0$ as\[\left(\ast_{k=1}^{m_1-1}[W_{1}^{k},W_{1}^{k+1}]\right)\ast \dots \ast\left(\ast_{k=1}^{m_j-1}[W_{j}^{k},W_{j}^{k+1}]\right)\ast \dots\ast \left(\ast_{k=1}^{m_n-1}[W_{n}^{k},W_{n}^{k+1}]\right)\]

To see $L'$ is homotopic to $L$, observe the set of intervals $[s^{k-1}_{j},s^{k}_{j}]$, and consequently the set of neighborhoods $W_{j}^{k}$, inherits an ordering from the natural ordering of $[0,1]$. Suppose the intervals $[s^{k-1}_{j},s^{k}_{j}]$ are ordered as $A_1,...,A_N\subseteq [0,1]$ and the neighborhoods $W_{j}^{k}$ are ordered as $w_1,...,w_N$ so that $L(A_\ell)\subseteq St(w_\ell,\nerveu)$. For $\ell=1,...,N-1$, let $b_{\ell}$ be the barycenter of $[w_{\ell},w_{\ell+1}]$. Find a partition $0=t_0<t_1<\dots <t_N=1$ such that $L'([t_{\ell-1},t_{\ell}])\subseteq St(w_\ell,\nerveu)$ for $\ell=1,...,N$ by choosing $t_{\ell}$ so that $L'(t_\ell)=b_{\ell}$ for $\ell=1,...,N$. Lemma \ref{homotopicpaths} applies and gives a basepoint preserving homotopy between $L$ and $L'$.

It now suffices to show the edge loops $L'$ and $E=[U_1,U_2]\ast [U_2,U_3]\ast \dots \ast[U_{n-1},U_n]$ are homotopic. We do so by showing $L'\ast \overline{E}$ is homotopic to a concatenation of inessential loops based at $U_0$. First, for each $j=1,...,n$, let $E_j=[U_1,U_2]\ast \dots \ast[U_{j-1},U_j]$. Observe $E_1=[U_1,U_1]$ is constant and $E_n=E$.\\
\begin{enumerate}
\item For $j=1,...,n$ and $k=1,...,m_j-1$, let \[a_{j}^{k}=E_j\ast [U_j,W_{j}^{k}]\ast [W_{j}^{k},W_{j}^{k+1}]\ast[W_{j}^{k+1},U_j]\ast \overline{E_{j}}.\] These loops are well-defined and inessential since $U_j\cap W_{j}^{k}\cap W_{j}^{k+1}\neq \emptyset$ determines a 2-simplex in $\nerveu$. Even in the case two or more of $U_j$, $W_{j}^{k}$, $W_{j}^{k+1}$ are equal, $a_{j}^{k}$ is still inessential.
\item For $j=1,...,n-1$, let \[b_{j}=E_j\ast [U_j,W_{j}^{m_j}]\ast [W_{j}^{m_j},U_{j+1}]\ast[U_{j+1},U_j]\ast \overline{E_{j}}.\] These loops are well-defined and inessential since $U_j\cap W_{j}^{m_j}\cap U_{j+1}\neq \emptyset$.
\end{enumerate}
Therefore the product \[p=\left(\ast_{k=1}^{m_1-1}a_{1}^{k}\right)\ast b_1 \ast \dots \ast \left(\ast_{k=1}^{m_j-1}a_{j}^{k}\right)\ast b_j \ast \dots \ast b_{n-1}\ast \left(\ast_{k=1}^{m_n-1}a_{n}^{k}\right)\]is inessential. We observe $p$ is homotopic to $L'\ast \overline{E}$ by reducing words. Indeed, for each $j=1,...,n-1$, $\left(\ast_{k=1}^{m_j-1}a_{j}^{k}\right)\ast b_j$ reduces to \[E_j\ast [U_j,W_{j}^{1}]\ast \left(\ast_{k=1}^{m_j-1}[W_{j}^{k},W_{j}^{k+1}]\right)\ast [W_{j}^{m_j},U_{j+1}]\ast \overline{E_{j}}\] and the last factor $\left(\ast_{k=1}^{m_n-1}a_{n}^{k}\right)$ reduces to $E_{n-1}\ast[U_n,W_{n}^{1}]\ast \left(\ast_{k=1}^{m_n-1}[W_{n}^{k},W_{n}^{k+1}]\right)\ast \overline{E_{n}}$ (recall that $W_{n}^{m_n}=U_0=U_n$). Using the fact that $W_{1}^{1}=U_1=U_n$ and $W_{j}^{m_j}=W_{j+1}^{1}$ to identify constant paths, it follows that $p$ reduces to $L' \ast \overline{E_n}=L'\ast \overline{E}$.
\end{proof}

\begin{theorem}\label{exactness1}
If $(\mathscr{U},U_0)\in \Lambda$ where $\mathscr{U}$ consists of
path connected sets and $p_{\mathscr{U}}:X\to |N(\mathscr{U})|$ is a
canonical map, then $p_{\mathscr{U}\ast}$ is surjective.
\end{theorem}

\begin{proof}
Since edge loops in $|\nerveu|$ represent all homotopy classes in $\pi_{1}(|\nerveu|,U_0)$, the theorem follows from Lemma \ref{surjectivity}.
\end{proof}

\subsection{$\Pi^{Sp}(\mathscr{U},x_0)\subseteq \ker p_{\mathscr{U}\ast}$}

The flavor of the proof of the following Lemma is quite similar to that of
Lemma \ref{surjectivity} but holds when the elements of $\mathscr{U}$ are
not necessarily path connected.

\begin{lemma}
\label{inclusion2} Let $\mathscr{U}$ be an open cover of $X$. Suppose $%
\gamma_1,\gamma_2$ are paths in $X$ such that $\gamma_1(1)=\gamma_2(0)$ and $%
\gamma_1(0)=\gamma_2(1)$. If $\gamma_i$ has image in $U_i\in\mathscr{U}$,
then $L=p_{\mathscr{U}}\circ (\gamma_1\ast\gamma_2)$ is an inessential loop
in $|N(\mathscr{U})|$.
\end{lemma}

\begin{proof}
We again use that $\scrv=\{V_U=p_{\scru}^{-1}(St(U,\nerveu))|U\in\scru\}$ is open refinement of $\scru$. Find partitions $0=s_{1}^{0}<s_{1}^{1}<\dots <s_{1}^{m_1}=\frac{1}{2}$ and $\frac{1}{2}=s_{2}^{0}<s_{2}^{1}<\dots <s_{2}^{m_2}=1$ and a sequence of neighborhoods $W_{1}^{1},...,W_{1}^{m_1},W_{2}^{1},...,W_{2}^{m_2}\in \scru$ such that $\gamma_i([s_{j}^{k-1},s_{j}^{k}])\subseteq V_{W_{j}^{k}}\subseteq W_{j}^{k}$ for $j=1,2$ and $k=1,...,m_j$. Since $\gamma_1(0)=\gamma_2(1)$, we may assume $W_{1}^{1}=W_{2}^{m_2}$ and since $\gamma_1(1)=\gamma_2(0)$, we may assume $W_{1}^{m_1}=W_{2}^{1}$. Observe that $L([s_{j}^{k-1},s_{j}^{k}])\subseteq St(W_{j}^{k},\nerveu)$.\\

Let $L'$ be the loop defined as the concatenation \[\left(\ast_{k=1}^{m_1-1}[W_{1}^{k},W_{1}^{k+1}]\right)\ast \left(\ast_{k=1}^{m_2-1}[W_{2}^{k},W_{2}^{k+1}]\right).\]

Find a partition \[0=t_{1}^{0}<t_{1}^{1}<\dots <t_{1}^{m_2}=t_{2}^{0}<t_{2}^{1}<\dots <t_{2}^{m_2}=1\] such that $L'([t_{j}^{k-1},t_{j}^{k}])\subseteq St(W_{j}^{k},\nerveu)$. By Lemma \ref{homotopicpaths}, $L$ and $L'$ are freely homotopic.\\

It now suffices to show $L'$ is inessential. We do so by showing $L'$ is homotopic to a concatenation of inessential loops. First, let $E_1=[W_{1}^{1},U_1]$ and $E_2=[W_{1}^{1},U_1]\ast[U_1,U_2]$.
\begin{enumerate}
\item For $k=1,...,m_1-1$, let \[a_{1}^{k}=E_1\ast [U_1,W_{1}^{k}]\ast [W_{1}^{k},W_{1}^{k+1}]\ast [W_{1}^{k+1},U_1]\ast \overline{E_{1}}.\]
\item Let $b_1=E_1\ast [U_1,W_{1}^{m_1}]\ast [W_{2}^{1},U_2]\ast [U_2,U_1]\ast \overline{E_{1}}$.
\item For $k=1,...,m_2-1$, let \[a_{2}^{k}=E_2\ast [U_2,W_{2}^{k}]\ast [W_{2}^{k},W_{2}^{k+1}]\ast [W_{2}^{k+1},U_2]\ast \overline{E_{2}}.\]
\item Let $b_2=E_2\ast [U_2,W_{1}^{1}]=[W_{1}^{1},U_1]\ast[U_1,U_2]\ast[U_2,W_{1}^{1}]$.
\end{enumerate}
Note the loop $a_{1}^{k}$ is well-defined and inessential since $U_1\cap W_{1}^{k}\cap W_{1}^{k+1}\neq \emptyset$ determines a 2-simplex in $\nerveu$. Similarly, the loops defined in 2.-4. are well-defined and inessential.

Consequently, the concatenation \[p=\left(\ast_{k=1}^{m_1-1}a_{1}^{k}\right)\ast b_1 \ast \left(\ast_{k=1}^{m_2-1}a_{2}^{k}\right)\ast b_2\] is an inessential loop based at $W_{1}^{1}$. We observe $p$ is homotopic to $L'$ by reducing words. Indeed, $\left(\ast_{k=1}^{m_1-1}a_{1}^{k}\right)\ast b_1$ reduces to $\left(\ast_{k=1}^{m_1-1}[W_{1}^{k},W_{1}^{k+1}]\right)\ast [W_{2}^{1},U_{2}]\ast \overline{E_{2}}$ and $\left(\ast_{k=1}^{m_2-1}a_{2}^{k}\right)\ast b_2$ reduces to $E_2\ast [U_2,W_{2}^{1}]\ast \left(\ast_{k=1}^{m_2-1}[W_{2}^{k},W_{2}^{k+1}]\right)$.
\end{proof}

\begin{theorem} \label{exactness2}
If $(\mathscr{U},U_0)\in \Lambda$ and $p_{\mathscr{U}}:X\to |N(\mathscr{U})|$ is a canonical map, then $\Pi^{Sp}(\mathscr{U},x_0)\subseteq
\ker p_{\mathscr{U}\ast}$.
\end{theorem}

\begin{proof}
Suppose $g=[\alpha][\gamma_1\ast\gamma_2][\reva]$ is a generator of $\tspu$ where $\gamma_i$ is a path with image in $U_i\in \scru$. Since $p_{\scru\ast}(g)=[p_{\scru}\circ \alpha][p_{\scru}\circ(\gamma_1\ast\gamma_2)][\overline{p_{\scru}\circ \alpha}]$ and $p_{\scru}\circ(\gamma_1\ast\gamma_2)$ is inessential by Lemma \ref{inclusion2}, $p_{\scru\ast}(g)$ is trivial in $\pi_{1}(|\nerveu|,U_0)$.
\end{proof}

\subsection{$\ker p_{\mathscr{U}\ast}\subseteq \Pi^{Sp}(\mathscr{U},x_0)$}

To show the inclusion $\ker p_{\mathscr{U}\ast}\subseteq \Pi^{Sp}(\mathscr{U}%
,x_0)$ holds under the hypotheses of Theorem \ref{shortexactsequence}, we
work to extend loops $|\partial\Delta_{2}|\to X$ to maps on the 1-skeleton of
subdivisions of $|\Delta_{2}|$.

\begin{definition}
\emph{An edge loop $L$ is \textit{short} if it is the concatenation $L=[v_0,v_1]\ast [v_1,v_2]\ast [v_2,v_3]$ of three linear paths. A \textit{Spanier edge loop} is an edge loop of the form $E\ast L\ast \overline{E}$ where $E$
is an edge path and $L$ is a short edge loop.}
\end{definition}

\begin{lemma}
\label{fishnet} Let $H$ be a subgroup of $\pi_{1}(X,x_0)$ and $n\geq 1$ be
an integer. Let $f:(sd^{n}|\Delta_{2}|)_1\to X$ be a based map on the
1-skeleton of $sd^{n}|\Delta_{2}|$, and $\beta=f|_{|\partial\Delta_{2}|}:|\partial\Delta_{2}|\to X$ be the restriction. If $[f\circ \sigma]\in H$ for
every Spanier edge loop $\sigma$ in $sd^{n}|\Delta_{2}|$, then $[\beta]\in H$.
\end{lemma}

\begin{proof}
It follows from elementary planar graph theory that $\pi_{1}((sd^{n}\dtr)_1,(0,0))$ is generated by the homotopy classes of Spanier edge loops. Thus if $[f\circ \sigma]\in H$ for every Spanier edge loop $\sigma$ in $sd^{n}\dtr$, then $f_{\ast}(\pi_{1}((sd^{n}\dtr)_1,(0,0)))\subseteq H$. In particular $[\beta]\in H$.
\end{proof}
Let $sd^{1}|\partial\Delta_{2}|$ be the first barycentric subdivision of $%
|\partial\Delta_{2}|$. The 0-skeleton consists of six vertices: Let $%
v_1,v_2,v_3$ be the vertices $(0,0),(1,0),(0,1)$ of $|\partial\Delta_{2}|$
respectively and $m_{i}$ be the vertex which is the barycenter of the edge
opposite $v_i$ for $i=1,2,3$.

\begin{definition}
\emph{ Let $\scru$ be an open cover of $X$. A map $%
\delta:|\partial\Delta_{2}|\to X$ is $\scru$\textit{-admissible} if
there are open neighborhoods $U_1,U_2,U_3\in \mathscr{U}$ such that $%
\bigcap_{i}U_i\neq \emptyset$ and $\delta(\overline{St}(v_i,sd^{1}|\partial%
\Delta_{2}|))\subseteq U_i$ for $i=1,2,3$. }
\end{definition}

\begin{remark}
\emph{The condition $\delta(\overline{St}(v_i,sd^{1}|\partial\Delta_{2}|))\subseteq U_i$ in the previous definition means precisely that $\delta([v_1,m_2]\cup [v_1,m_3])\subseteq U_1$, $\delta([v_2,m_1]\cup
[v_2,m_3])\subseteq U_2$, and $\delta([v_3,m_1]\cup [v_3,m_2])\subseteq U_3$%
. }
\end{remark}

\begin{lemma}
\label{uadmissible} Let $\mathscr{U}$ be an open cover of $X$ consisting of
path connected sets. If $\delta:|\partial\Delta_{2}|\to X$ is $\mathscr{U}$%
-admissible and $\alpha:[0,1]\to X$ is a path from $x_0$ to $\delta(v_1)$,
then $[\alpha][\delta][\reva]\in \Pi^{Sp}(\mathscr{U},x_0)$.
\end{lemma}

\begin{proof}
Since $\delta:\pdtr\to X$ is $\scru$-admissible there are sets $U_1,U_2,U_3\in \scru$ such that $\bigcap_{i}U_i$ contains a point $z$ and $\delta(\overline{St}(v_i,\pdtr))\subseteq U_i$ for $i=1,2,3$.

Whenever $j,k\in \{1,2,3\}$ and $j\neq k$, let $\gamma_{j,k}$ be the path which is the restriction of $\delta$ to $[v_j,m_k]$. With this notation \[[\delta]= [\gamma_{1,3}][\overline{\gamma_{2,3}}][\gamma_{2,1}][\overline{\gamma_{3,1}}][\gamma_{3,2}] [\overline{\gamma_{1,2}}].\] Note $Im(\gamma_{j,k})\subseteq U_j$ for all choice of $j,k$. Thus if $j,k,l\in \{1,2,3\}$ are distinct, the endpoint of $\gamma_{j,k}$ lies in $U_j\cap U_l$.

Define six paths in $X$: Whenever $j,k\in \{1,2,3\}$ and $j\neq k$, let $\eta_{j,k}$ be a path in $U_j$ from $\delta(m_k)$ to $z$. Such paths are guaranteed to exist since each element of $\scru$ is path connected. Now given any path $\alpha:[0,1]\to X$ from $x_0$ to $\delta(v_1)$ let 
\[\begin{array}{rcl}
\zeta_{1} & = & [\alpha][\gamma_{1,3}\ast\eta_{1,3}\ast\overline{\eta_{1,2}}\ast\overline{\gamma_{1,2}}][\reva]\\
\zeta_{2} & = & [\alpha\ast\gamma_{1,2}\ast\eta_{1,2}][\overline{\eta_{1,3}}][\eta_{2,3}][\overline{\alpha\ast\gamma_{1,2}\ast \eta_{1,2}}]\\
\zeta_{3} & = & [\alpha\ast\gamma_{1,2}\ast\eta_{1,2}][\overline{\eta_{2,3}}\ast\overline{\gamma_{2,3}}\ast\gamma_{2,1}\ast\eta_{2,1}] [\overline{\alpha\ast\gamma_{1,2}\ast\eta_{1,2}}] \\
\zeta_{4} & = & [\alpha\ast\gamma_{1,2}\ast\eta_{1,2}][\overline{\eta_{2,1}}][\eta_{3,1}] [\overline{\alpha\ast\gamma_{1,2}\ast\eta_{1,2}}]\\
\zeta_{5} & = &  [\alpha\ast\gamma_{1,2}\ast\eta_{1,2}][\overline{\eta_{3,1}}\ast\overline{\gamma_{3,1}}\ast\gamma_{3,2}\ast\eta_{3,2}] [\overline{\alpha\ast\gamma_{1,2}\ast\eta_{1,2}}]\\
\zeta_{6} & = & [\alpha\ast\gamma_{1,2}][\eta_{1,2}][\overline{\eta_{3,2}}] [\overline{\alpha\ast\gamma_{1,2}}]
\end{array}\]
Note each $\zeta_i$ is written as a product to illustrate that $\zeta_1$, $\zeta_3$, $\zeta_5\in \spu$ and $\zeta_2$, $\zeta_4$, $\zeta_6\in \tspu$. Since $\spu\subseteq \tspu$, we have $\zeta_i\in \tspu$ for each $i$. A straightforward check gives $[\alpha][\delta][\reva]=\zeta_1\zeta_2\zeta_3\zeta_4\zeta_5\zeta_6$ and thus $[\alpha][\delta][\reva]\in \tspu$.
\end{proof}

\begin{lemma}
\label{admissibletheorem} If $f:(sd^{n}|\Delta_{2}|)_1\to X$ is a map such
that the restriction of $f$ to the boundary $\partial\tau$ of every
2-simplex $\tau$ in $sd^{n}|\Delta_{2}|$ is $\mathscr{U}$-admissible and $\beta=f|_{|\partial\Delta_{2}|}$, then $[\beta]\in \Pi^{Sp}(\mathscr{U},x_0)$.
\end{lemma}

\begin{proof}
By Lemma \ref{fishnet} it suffices to show $[f\circ\sigma]\in \tspu$ for any Spanier edge loop $\sigma$ in $sd^{n}\dtr$. Suppose $\sigma=E\ast L\ast \overline{E}$ is such a Spanier edge loop where $E$ is an edge path from $(0,0)$ to a vertex of $\partial\tau$ and $L$ is a short edge loop traversing $\partial\tau$. Let $\alpha=f\circ E$ and $\delta=f\circ L$. Since $\delta$ is $\scru$-admissible by assumption, $[f\circ \sigma]=[\alpha][\delta][\reva]\in\tspu$ by Lemma \ref{uadmissible}.
\end{proof}

\begin{lemma}
\label{technical} Suppose $(\mathscr{U},U_0)\in \Lambda$ where $\mathscr{U}$ consists of path connected sets and $p_{\mathscr{U}}:X\to |N(\mathscr{U})|$ is a canonical map. If $\beta:|\partial\Delta_{2}|\to X$ is a loop such that $[\beta]\in \ker p_{\mathscr{U}\ast}$, then there exists an
integer $n\geq 1$ and an extension $f:(sd^n|\Delta_{2}|)_1\to X$ of $\beta$
such that the restriction of $f$ to the boundary $\partial\tau$ of every
2-simplex $\tau$ in $sd^{n}|\Delta_{2}|$ is $\mathscr{U}$-admissible.
\end{lemma}

\begin{proof}
Since $p_{\scru}\circ \beta:\pdtr\to |\nerveu|$ is a null-homotopic loop based at the vertex $U_0$, it extends to a map $h:\dtr\to |\nerveu|$. Since, by the definition of $p_{\scru}$, we have $V_U=p_{\scru}^{-1}(St(U,\nerveu))\subseteq U$ for each $U\in \scru$, the cover $\scrv=\{V_U|U\in\scru\}$ is an open refinement of $\scru$. Additionally, if $W_U=h^{-1}(St(U,\nerveu))$, the collection $\mathscr{W}=\{W_U|U\in \scru\}$ is an open cover of $\dtr$. 

Following Theorem 16.1 (finite simplicial approximation) in \cite{Mu84}, find a simplicial approximation for $h$ using the cover $\mathscr{W}$: let $\lambda$ be a Lebesgue number for $\mathscr{W}$ so that any subset of $\dtr$ of diameter less than $\lambda$ lies in some element of $\mathscr{W}$. Choose an integer $n$ such that each simplex in $sd^{n}\dtr$ has diameter less than $\lambda/2$. Thus the star $St(a,sd^{n}\dtr)$ of each vertex $a$ in $sd^{n}\dtr$ lies in a set $W_{U_a}$ for some $U_a\in \scru$. The assignment $a\mapsto U_a$ on vertices extends to a simplicial approximation $h':sd^{n}\dt\to \nerveu$ of $h$, i.e. a simplicial map $h'$ such that \[h(St(a,sd^{n}\dt))\subseteq St(h'(a),\nerveu)=St(U_a,\nerveu)\] for each vertex $a$ \cite[Lemma 14.1]{Mu84}.

We construct an extension $f:(sd^{n}\dtr)_1\to X$ of $\beta$ such that the restriction of $f$ to the boundary of each 2-simplex of $sd^{n}\dtr$ is $\scru$-admissible. First, define $f$ on vertices: for each vertex $a\in (sd^{n}\dtr)_0$, pick a point $f(a)\in U_a$. In particular, if $a$ is a vertex of the subcomplex $sd^{n}\pdtr$ of $(sd^{n}\dtr)_{1}$, take $f(a)=\beta(a)$. This choice is well defined since whenever $a\in sd^{n}\pdtr$, we have $p_{\scru}\circ \beta(a)=h(a)\in St(U_a,\nerveu)$ and thus $\beta(a)\in V_{U_a}\subseteq U_a$.

Observe if $[a,b,c]$ is any 2-simplex in $sd^{n}\dt$, then $h'([a,b,c])$ is a simplex of $\nerveu$ spanned by the set of vertices $\{U_a,U_b,U_c\}$ (possibly containing repetitions). By the construction of $\nerveu$, the intersection $U_a\cap U_b\cap U_c$ is non-empty. 

Extend $f$ to the interiors of 1-simplices of $(sd^{n}|\Delta_{2}|)_1$; to define $f$ on a given 1-simplex $[a,b]$, consider two cases:

Case I: Suppose 1-simplex $[a,b]$ is the intersection of two 2-simplices in $sd^{n}|\Delta_{2}|$. Extend $f$ to the interior of $[a,b]$  by taking $f|_{[a,b]}$ to be a path in $U_a\cup U_b$ from $f(a)$ to $f(b)$ such that if $m$ is the barycenter of $[a,b]$, then $f|_{[a,b]}([a,m])\subseteq U_a$ and $f|_{[a,b]}([m,b])\subseteq U_b$. Such a path is guaranteed to exist since $U_a$, $U_b$ are path connected and have non-trivial intersection.

Case II: Suppose 1-simplex $[a,b]$ is the face of a single 2-simplex or equivalently that $[a,b]$ is a 1-simplex of $sd^{n}\pdtr$. In this case, let $f|_{[a,b]}=\beta_{[a,b]}$. Again, note $U_a\cap U_b\cap U_c\neq \emptyset$. Since $St(a,sd^{n}\dtr)\subseteq W_{U_a}$ and $St(b,sd^{n}\dtr)\subseteq W_{U_b}$, we have \[p_{\scru}\circ \beta(int([a,b]))=h(int([a,b]))\subseteq St(U_a,\nerveu)\cap St(U_{b},\nerveu).\]Thus $\beta(int([a,b]))\subseteq V_{U_a}\cap V_{U_b}\subseteq U_a\cap U_{b}$. Note if $m$ is the barycenter of $[a,b]$, then $f([a,m])\subseteq U_a$ and $f([m,b])\subseteq U_b$. 

Now $f:(sd^{n}\dtr)_1\to X$ is a map extending $\beta$ such that the restriction of $f$ to the boundary $\partial\tau$ of every 2-simplex $\tau$ in $sd^{n}\dtr$ is $\scru$-admissible.
\end{proof}

\begin{theorem}
\label{exactness3} If $(\mathscr{U},U_0)\in \Lambda$ where $\mathscr{U}$ consists of path connected sets and $p_{\mathscr{U}}:X\to |N(%
\mathscr{U})|$ is a canonical map, then $\ker p_{\mathscr{U}\ast}\subseteq
\Pi^{Sp}(\mathscr{U},x_0)$.
\end{theorem}

\begin{proof}
The theorem follows directly from Lemma \ref{admissibletheorem} and Lemma \ref{technical}.
\end{proof}

\section{Characterizations of $\ker\Psi_X$}

Theorem \ref{shortexactsequence} allows us to characterize the kernel of the
natural map $\Psi_{X}:\pi_{1}(X,x_0)\to \check{\pi}_{1}(X,x_0)$ as the
Spanier group $\pi^{Sp}(X,x_0)$.

\begin{theorem}
\label{leftexactsequence} If $X$ is a locally path connected, paracompact
Hausdorff space, then there is a natural exact sequence 
\begin{equation*}
\xymatrix{ 1 \ar[r] & \spx \ar[r] & \pionex \ar[r]^-{\Psi_{X}} & \shape}
\end{equation*}
where the second map is inclusion.
\end{theorem}

\begin{proof}
It suffices to show $\ker\Psi_{X}=\spx$. The inclusion $\spx \subseteq \ker \Psi_{X}$ holds for arbitrary $X$; this fact follows from Propositions \ref{spanierinclusion} and \ref{inclusion1} and is proved directly in \cite{FZ07}. Since $X$ is paracompact Hausdorff, $\spx=\tspx$ by Theorem \ref{paracompact1}. Thus, it suffices to show $\ker\Psi_{X}\subseteq \tspx$. 

Suppose $[\beta]\in \ker\Psi_{X}$ (equivalently $[\beta]\in \ker p_{\scrw\ast}$ for every $(\scrw,W_0)\in \Lambda$) and $\scru$ is a given open cover of $X$. It suffices to show $[\beta]\in \tspu$. Let $V_0$ be a path connected neighborhood of $x_0$ contained in $W_0$. Every point $x\in X\backslash V_0$ is an element of some $U_x\in \scru$. Let $V_x$ be a path connected neighborhood of $x$ contained in $U_x$ which is disjoint from $\{x_0\}$. Now $\scrv=\{V_0\}\cup\{V_x|x\in X\backslash V_0\}$ is an open cover of $X$ which contains path connected sets, is a refinement of $\scru$, and such that there is a single set $V_0$ having the basepoint $x_0$ as an element. Since $X$ is paracompact Hausdorff, $\scrv$ is normal and therefore $(\scrv,V_0)\in \Lambda$.

Since $[\beta]\in \ker p_{\scrv\ast}$ and $\ker p_{\scrv\ast}\subseteq \tsp(\scrv,x_0)$ by Theorem \ref{shortexactsequence}, we have $[\beta]\in \tsp(\scrv,x_0)$. Note that $\tsp(\scrv,x_0)\subseteq \tspu$ since $\scrv$ refines $\scru$. Thus $[\beta]\in \tspu$.

Regarding naturality, it is well-known that $\Psi_{X}$ is natural in $X$. Thus it suffices to check that if $f:X\to Y$, $f(x_0)=y_0$ is a map, then $f_{\ast}(\spx)\subseteq \pi^{Sp}(Y,y_0)$. This follows directly from Remark \ref{naturalityofspanier}.
\end{proof}
%
%

\begin{corollary}
\label{shapevanishing} If $X$ is a locally path connected, paracompact
Hausdorff space, then $X$ is $\pi_1$-shape injective if and only if $%
\pi^{Sp}(X,x_0)=1$.
\end{corollary}

\begin{example}
\label{example1} \emph{\ Theorem \ref{leftexactsequence} fails to hold in
the non-locally path connected case. A counterexample is the compact space $Z\subset\mathbb{R}^{3}$ of \cite{FRVZ11} obtained by rotating the closed
topologist's sine curve so the linear path component forms a cylinder and
connecting the two resulting surface components by attaching a single
arc. We have $\pi_{1}(Z,z_0)\cong \mathbb{Z}$ and $\pi^{Sp}(Z,z_0)=1$, yet $
\ker\Psi_{Z}=\pi_{1}(Z,z_0)$. }
\end{example}

The identification $\pi^{Sp}(X,x_0)=\ker\Psi_X$ in Theorem \ref%
{leftexactsequence} allows us to give two alternative characterizations of $%
\ker\Psi_X$: one in terms of $\mathscr{U}$-homotopy (from Section \ref{sectionuhomotopy}) and one in terms of covering spaces of $X$. Here covering spaces and maps are meant in the classical sense \cite{Spanier66}.

\begin{corollary}
\label{uhomotopycharacterization} Suppose $X$ is a locally path connected,
paracompact Hausdorff space and $\alpha:[0,1]\to X$ is a loop based at $x_0$. The following are equivalent:

\begin{enumerate}
\item $[\alpha]\in\ker\Psi_{X}$

\item For every open cover $\mathscr{U}$ of $X$, $\alpha$ is null-$%
\mathscr{U}$-homotopic.

\item For every covering map $r:Y\to X$, $r(y_0)=x_0$, the unique lift $%
\tilde{\alpha}:[0,1]\to Y$ of $\alpha$ (i.e. such that $r\circ\tilde{\alpha}%
=\alpha$) starting at $y_0$ is a loop in $Y$.
\end{enumerate}
\end{corollary}

\begin{proof}
(1. $\Leftrightarrow$ 2.) By Theorems \ref{paracompact1} and \ref{leftexactsequence}, we have $\tspx=\spx=\ker\Psi_X$. Since $X$ is locally path connected, every open cover $\scru$ admits an open refinement $\scrv$ whose elements are path connected and thus $\tspv=\nulhv$ by Theorem \ref{characterizethickspangrp}. Since $\tspv=\nulhv$ for all $\scrv$ in a cofinal subset of $\ox$ the equality $\tspx=\bigcap_{\scru}\nulhu$ follows.

(1. $\Leftrightarrow$ 3.) Since $\spx=\ker\Psi_X$, we show $\spx$ is the intersection of all images $r_{\ast}(\pi_{1}(Y,y_0))$ where $r:Y\to X$, $r(y_0)=x_0$ is a covering map. This equality follows directly from the following well-known result from the covering space theory of locally path connected spaces (see Lemma 2.5.11 and Theorem 2.5.13 in \cite{Spanier66}): given a subgroup $H$ of $\pionex$, there is an open cover $\scru$ of $X$ such that $\spu\subseteq H$ if and only if there is a covering map $r:Y\to X$, $r(y_0)=x_0$ with $r_{\ast}(\pi_{1}(Y,y_0))=H$.
\end{proof}

\begin{example}
\label{example2} \emph{\ Consider the Peano continua $Y^{\prime}$ and $%
Z^{\prime}$ of \cite{FRVZ11} constructed as subspaces of $\mathbb{R}^3$
(these spaces also appear as the spaces $A$ and $B$ in \cite{CMRZZ08}
respectively). The space $Y^{\prime}$ is constructed by rotating the
topologists sine curve $T=\{(x,0,\sin(1/x))|0<x\leq 1\}\cup \{0\}\times
\{0\}\times [0,1]$ about the ``central axis" (i.e. the portion $\{0\}\times
\{0\}\times [0,1]$ of the z-axis) and attaching horizontal arcs so $%
Y^{\prime}$ is locally path connected and so the arcs become dense only on
the central axis. The space $Z^{\prime}$ is an inverted version of $%
Y^{\prime}$ in the sense that $T$ is now rotated around the vertical line
passing through $(1,0,\sin 1)$ and arcs are attached becoming dense around
the outside cylinder. Take the basepoint $z_0$ of $Z^{\prime}$ to lie on the
outer cylinder.\newline
\indent The authors of \cite{FRVZ11} observe both spaces have non-vanishing
Spanier group and thus fail $\pi_1$-shape injectivity. The Spanier group of $%
Y^{\prime}$ is known to be non-trivial \cite[Prop 3.2(1)]{FRVZ11}, however
slightly more effort is required to show a simple closed curve $L$, based at $z_0$, traversing the outer cylinder of $Z^{\prime}$ is homotopically non-trivial 
\cite[Lemma 3.1]{CMRZZ08} and satisfies $[L]\in \pi^{Sp}(Z^{\prime},z_0)$.
To do so, one can, given an open cover $\mathscr{U}$ of $X$, look for small
arcs in the intersections of elements of $\mathscr{U}$ and factor $[L]$ as a
product of generators of $\pi^{Sp}(\mathscr{U},z_0)$.\newline
\indent Using the above results, we can observe $[L]\in
\pi^{Sp}(Z^{\prime},z_0)$ without having to construct explicit
factorizations and without having to rely on convenient covers with
``enough" arcs lying in intersections. For instance, to see $[L]\in
\pi^{Sp}(Z^{\prime},z_0)$, let $\mathscr{U}$ be an open cover of $Z^{\prime}$
and find a subdivision of $0=t_0<t_1<\dots t_n=1$ so $L([t_{i-1},t_i])%
\subseteq U_i$ for $U_i\in \mathscr{U}$ and $U_1=U_n$. There is an arc $%
\alpha$ connecting $L(0)$ to the surface portion with image in $U_1$ and a
loop $L^{\prime}$ on the surface portion based at $\alpha(1)$ such that $%
L^{\prime}([t_{i-1},t_i])\subseteq U_i$ for each $i$. Thus $L$ is $%
\mathscr{U}$-homotopic to the homotopically trivial loop $\alpha\ast
L^{\prime}\ast \reva$. Since $[L]$ is null-$\mathscr{U}$-homotopic for
any given $\mathscr{U}$, we have $[L]\in
\ker\Psi_{X}=\pi^{Sp}(Z^{\prime},z_0)$ by Corollary \ref%
{uhomotopycharacterization}. }
\end{example}

\begin{remark}
\emph{\ If, in addition to the hypotheses on $X$ in Theorem \ref%
{leftexactsequence}, $\Lambda$ admits a cofinal, directed
subsequence $(\mathscr{U}_n,U_n)$ where $(\mathscr{U}_{k+1},U_{k+1})$
refines $(\mathscr{U}_{k},U_{k})$ (for instance, when $X$ is compact
metric), the exact sequence in Theorem \ref{leftexactsequence} extends to
the right via the first derived limit $\varprojlim^{1}$ for non-abelian
groups \cite[Ch. II, \textsection 6.2]{MS82}. In doing so, we obtain a
connecting function $\delta$ and an exact sequence 
\begin{equation*}
\xymatrix{ 1 \ar[r] & \spx \ar[r] & \pionex \ar[r]^-{\Psi_{X}} & \shape
\ar[r]^-{\delta} & \varprojlim^{1}\tsp(\scru_n,x_0) \ar[r] & \ast }
\end{equation*}
in the category of pointed sets. Thus elements in the image of $\Psi_{X}$
are precisely those mapped to basepoint of the first derived limit of the
inverse sequence 
\begin{equation*}
...\subseteq\Pi^{Sp}(\mathscr{U}_{3},x_0)\subseteq \Pi^{Sp}(\mathscr{U}%
_2,x_0)\subseteq\Pi^{Sp}(\mathscr{U}_1,x_0)
\end{equation*}
of thick Spanier groups. }
\end{remark}
\section{A shape group topology on $\pi_{1}(X,x_{0})$}

For a general space $X$, the fundamental group $\pi_{1}(X,x_{0})$ admits a
variety of distinct natural topologies \cite{Brazas,BDLM,Fabel2}, one of which is the pullback from the product of discrete groups \cite{Fabel1,Melikhov} described as follows.

By definition, the geometric realization of a simplicial complex is locally
contractible. Hence the inverse limit space $\check{\pi}_{1}(X,x_{0})$
admits a natural topology as a subspace of the product of discrete groups $\Pi _{\Lambda}\pi _{1}(|N(\mathscr{U})|,U_{0}).$ 

To impart a topology on $\pi _{1}(X,x_{0})$, recall the homomorphism $\Psi
_{X}:\pi _{1}(X,x_{0})\rightarrow \check{\pi}_{1}(X,x_{0})$ and declare the
open sets of $\pi _{1}(X,x_{0})$ to be precisely sets of the form $\Psi
_{X}^{-1}(W)$ where $W$ is open in $\check{\pi}_{1}(X,x_{0}).$ We refer to
this group topology on $\pi_{1}(X,x_{0})$ as the \textit{shape topology}. 

The space $\check{\pi}_{1}(X,x_{0})$ is Hausdorff (since $\check{\pi}_{1}(X,x_{0})$ is a subspace of the arbitrary product of Hausdorff spaces).
Hence, if $\psi _{X}$ is one-to-one, then $\pi _{1}(X,x_{0})$ is Hausdorff
(since in general the preimage under a continuous injective map of a
Hausdorff space is Hausdorff). Conversely, if $\psi _{X}$ is not one-to-one,
then $\pi _{1}(X,x_{0})$ is not Hausdorff (since, with the pullback topology, if $x\neq y$ and $\psi_{X}(x)=\psi _{X}(y)$ then $x$ and $y$ cannot be separated by open sets).
The foregoing is summarized as follows.

\begin{proposition}
The following are equivalent for any space $X$:

\begin{enumerate}
\item  $X$ is $\pi_1$-shape injective.

\item  $\pi _{1}(X,x_{0})$ is Hausdorff.
\end{enumerate}
\end{proposition}

We characterize a basis for the shape topology using the short exact
sequences from Section \ref{mainsection}.

\begin{remark}
\label{opensubgroup} \emph{\ Note if $G$ is a topological group, $H$ is a
subgroup of $G$, and $K$ is a subgroup of $H$ which is open in $G$, then $H$
is also open in $G$ since $H$ decomposes as a union of open cosets of $K$. }
\end{remark}

\begin{lemma}
\label{tspanieropen2} If $X$ is locally path connected, paracompact
Hausdorff, then for every open cover $\mathscr{U}$ of $X$, both $\pi^{Sp}(\mathscr{U},x_0)$ and $\Pi^{Sp}(\mathscr{U},x_0)$ are open in $\pi_{1}(X,x_0)$.
\end{lemma}

\begin{proof}
Suppose $\scru$ is an open cover of $X$. Pick a set $U_0\in \scru$ such that $x_0\in U_0$. Since $X$ is locally path connected, paracompact Hausdorff, there exists $(\scrv,V_0)\in \Lambda$ refining $(\scru,U_0)$ such that $\scrv$ consists of path connected sets (e.g. see the the second paragraph in the proof of Theorem \ref{leftexactsequence}). Let $p_{\scrv}:X\to |N(\scrv)|$ be a based canonical map. Recall $\pi_{1}(|N(\scrv)|,V_0)$ is discrete and observe $p_{\scrv\ast}:\pi_{1}(X,x_0)\to \pi_{1}(|N(\scrv)|,V_0)$ is continuous. Additionally, $\tsp(\scrv,V_0)=\ker p_{\scrv\ast}$ by Theorem \ref{shortexactsequence}. Thus $\tsp(\scrv,V_0)$ is an open subgroup of $\pionex$. Since $\scrv$ refines $\scru$, we have $\tspv\subseteq \tspu$ and thus $\tspu$ is open by Remark \ref{opensubgroup}.

Since $X$ is paracompact Hausdorff, there is an open refinement $\scrw$ of $\scru$ such that $\tspw\subseteq\spu$ (See the proof of Theorem \ref{paracompact1}). By the previous paragraph, $\tspw$ is open. Thus $\spu$ is open by Remark \ref{opensubgroup}.
\end{proof}

\begin{theorem}
\label{thickspanierbasis} If $X$ is locally path connected, paracompact
Hausdorff, the set of thick Spanier groups $\Pi^{Sp}(\mathscr{U},x_0)$ form
a neighborhood base at the identity of $\pi_{1}(X,x_0)$.
\end{theorem}

\begin{proof}
The hypotheses on $X$ allow us to replace $\Lambda$ with the cofinal directed subset $\Lambda '$ consisting of pairs $(\scru,U_0)$ such that $\scru$ contains only path connected sets. We alter notation for brevity. For $\lambda=(\scru,U_0)\in \Lambda '$, let $e_{\lambda}$ be the identity of $G_{\lambda}=\pi_{1}(|\nerveu|,U_0)$. Additionally, let $S_{\lambda}=\tspu$, and $p_{\lambda}=p_{\scru\ast}:\pionex\to G_{\lambda}$. If $\lambda '\geq \lambda$ in $\Lambda '$, let $p_{\lambda\lambda '}:G_{\lambda '}\to G_{\lambda}$ be homomorphism induced by the projection. Thus $\shape\cong\varprojlim (G_{\lambda},p_{\lambda\lambda '},\Lambda ')\subseteq \prod_{\Lambda '}G_{\lambda}$ and $\Psi_{X}(\alpha)=(p_{\lambda}(\alpha))$.\\
\indent In $\shape$, a basic open neighborhood of the identity is of the form \[W_F=\shape \cap\left(\prod_{\lambda\in F}\{e_{\lambda}\} \times \prod_{\lambda\in \Lambda '-F}G_{\lambda}\right)\]where $F$ is a finite subset of $\Lambda '$. Therefore the sets $\Psi_{X}^{-1}(W_F)$ form a neighborhood base at the identity of $\pionex$.

By Theorem \ref{shortexactsequence}, $S_{\lambda}=\ker p_{\lambda}$ for each $\lambda\in \Lambda '$. Thus $\Psi_{X}^{-1}(W_F)=\bigcap_{\lambda\in F}S_{\lambda}$ for each finite set $F\subset \Lambda '$. For given $F$, take $\mu\in \Lambda '$ such that $\mu$ refines $\lambda$ for each $\lambda\in F$. This gives $S_{\mu}\subseteq \Psi_{X}^{-1}(W_F)$. Consequently, the open subgroups $S_{\lambda}$ for a neighborhood base at the identity.
\end{proof}

\begin{corollary}
\label{spanierbasis} If $X$ is locally path connected, paracompact
Hausdorff, the set of Spanier groups $\pi^{Sp}(\mathscr{U},x_0)$ form a
neighborhood base at the identity of $\pi_{1}(X,x_0)$.
\end{corollary}

\begin{proof}
By Theorem \ref{thickspanierbasis}, the thick Spanier groups $\Pi^{Sp}(\mathscr{U},x_0)$ form a neighborhood base at the identity of $\pi_{1}(X,x_0)$. For every open cover $\scru$ of $X$, the Spanier group $\pi^{Sp}(\mathscr{U},x_0)$ is an open subgroup of $\pi_1(X,x_0)$ by Lemma \ref{tspanieropen2} and $\pi^{Sp}(\mathscr{U},x_0)\subseteq \Pi^{Sp}(\mathscr{U},x_0)$ by Proposition \ref{tspanier1}. Thus the Spanier groups $\pi^{Sp}(\mathscr{U},x_0)$ form a neighborhood base at the identity.
\end{proof}
The basis in Corollary \ref{spanierbasis} shows the shape topology on $%
\pi_{1}(X,x_0)$ consists precisely of the data of the covering space theory
of $X$.

\begin{theorem}
\label{coveringclassification} Suppose $X$ is locally path connected,
paracompact Hausdorff and $H$ is a subgroup of $\pi _{1}(X,x_{0})$. The
following are equivalent:

\begin{enumerate}
\item  $H$ is open in $\pi_{1}(X,x_0)$.

\item  There is an open cover $\mathscr{U}$ of $X$ such that $\pi^{Sp}(\mathscr{U},x_0)\subseteq H$.

\item  There is a covering map $r:Y\to X$, $r(y_0)=x_0$ such that $r_{\ast}(\pi_{1}(Y,y_0))=H$.
\end{enumerate}
\end{theorem}

\begin{proof}
(1. $\Leftrightarrow$ 2.) follows directly from Corollary \ref{spanierbasis} and Remark \ref{opensubgroup}. (2. $\Leftrightarrow$ 3.) is a well-known result from covering space theory (used above in the proof of Corollary \ref{uhomotopycharacterization}).
\end{proof}
Thus the well-known classification of covering spaces can be extended to
non-semilocally simply connected spaces in the following way: If $X$ is
locally path connected, paracompact Hausdorff, then there is a canonical
bijection between the equivalence classes of connected coverings of $X$ (in
the classical sense) and conjugacy classes of open subgroups of $%
\pi_{1}(X,x_0)$ with the shape topology.

\end{document}